\documentclass[12pt]{article}
\usepackage{amsmath,amssymb,amsthm,amsfonts,hhline,color}
\usepackage{mathrsfs}
\usepackage{bbm}
\usepackage[titletoc,title]{appendix}
\usepackage{hyperref}
\usepackage{comment}
\usepackage{color}
\newcommand{\ds}{\displaystyle}

\newcommand{\dfr}[2]{\dfrac{#1}{#2}}%ams
\newcommand{\cd}{\cdot}
\newcommand{\lam}{\lambda}
\newcommand{\cds}{\cdots}
\newcommand{\dsum}{\displaystyle \sum}
\newcommand{\ol}[1]{\overline{#1}}

\renewcommand{\l}{\left}
\renewcommand{\r}{\right}

\newcommand{\q}{\quad}

\newcommand{\la}{\langle}
\newcommand{\ra}{\rangle}

\newcommand{\abs}[1]{\lvert{#1}\rvert}
\newcommand{\Z}{\mathbb{Z}}
\newcommand{\C}{\mathbb{C}}
\newcommand{\R}{\mathbb{R}}
\newcommand{\N}{\mathbb{N}}

\newcommand{\aut}{\mathrm{Aut}}

\newcommand{\tr}{\mathrm{Tr}}

\newcommand{\id}{\mathrm{id}}

\newcommand{\al}{\alpha}

\newcommand{\w}{\omega}

\newcommand{\fs}{\mathfrak{S}}
\newcommand{\vac}{\mathbbm{1}}

\newcommand{\mymid}{\,|\,}

\makeatletter
\@addtoreset{equation}{section}

\makeatother

\theoremstyle{plain}
\newtheorem{mthm}{Theorem}
\newtheorem{thm}{Theorem}[section]
\newtheorem{prop}[thm]{Proposition}
\newtheorem{lem}[thm]{Lemma}
\newtheorem{cor}[thm]{Corollary}

\theoremstyle{definition}
\newtheorem{df}[thm]{Definition}

\newtheorem{cond}{Condition}
\newtheorem{nota}[thm]{Notation}
\newtheorem{rem}[thm]{Remark}

\newcommand{\sfr}[2]{\leavevmode\kern-.05em
  \raise.5ex\hbox{\the\scriptfont0 #1}\kern-.1em
  /\kern-.15em\lower.25ex\hbox{\the\scriptfont0 #2}\kern.02em}

\newcommand{\shf}{\sfr{1}{2}}
\newcommand{\sshf}{\leavevmode\kern-.05em
  \raise.3ex\hbox{\the\scriptfont0 1}\kern-.1em
  /\kern-.10em\lower.5ex\hbox{\the\scriptfont0 2}\kern.02em}

\DeclareMathOperator*{\tensor}{\otimes}

\newcommand{\pf}{\noindent {\bf Proof:}\q }

\newcommand{\com}{\mathrm{Com}}

\renewcommand{\ker}{\mathrm{Ker}}

\newcommand{\longto}{\longrightarrow}

\title{The classification of vertex operator algebras of OZ-type
generated by Ising vectors of $\sigma$-type}

\author{
  Cuipo  Jiang\thanks{Partially supported by NSFC grants 12171312, 11771281.}
  \medskip\\
  \textit{\small School of Mathematical Science,  Shanghai Jiao Tong University} \\
  \textit{\small Shanghai, 200240, China}
  \\
  {\small email: \texttt{cpjiang@sjtu.edu.cn}}
  \bigskip\\
  Ching Hung Lam\thanks{Partially supported by NSTC grant  110-2115-M-001-003-MY3  of Taiwan.}%
  \medskip\\
  \textit{\small Institute of Mathematics, Academia Sinica, Taipei, Taiwan 10617}\\
  {\small e-mail: \texttt{chlam@math.sinica.edu.tw}}
  \bigskip\\
  Hiroshi Yamauchi\footnote{Partially supported by JSPS KAKENHI Grant Numbers JP19K03409, JP19KK0065,
  JP22K03252.}
  \medskip\\
  \textit{\small Department of Mathematics,  Tokyo Woman's Christian University}\\
  \textit{\small 2-6-1 Zempukuji, Suginami-ku, Tokyo 167-8585, Japan}\\
  {\small e-mail: \texttt{yamauchi@lab.twcu.ac.jp}}
}
\date{}

\begin{document}

\maketitle

\begin{abstract}
We classify vertex operator algebras (VOAs) of OZ-type generated by Ising vectors
of $\sigma$-type.
As a consequence of the classification, we also prove that such VOAs are simple, rational,
$C_2$-cofinite and unitary, that is, they have compact real forms generated by Ising vectors
of $\sigma$-type over the real numbers.
\end{abstract}

\pagestyle{plain}
%%%%%%    TEXT  START   %%%%%%%%

\baselineskip 6mm

%\tableofcontents

\section{Introduction}
This article is a continuation of our work \cite{JLY} on characterizing vertex operator algebras
of OZ-type generated by Ising vectors of $\sigma$-type.
A vertex operator algebra (VOA for short) is of OZ-type
(stands for One-Zero) if it has the grading $V=\oplus_{n=0}^{\infty} V_n$ such that $\dim V_0=1$
and $\dim V_1=0$.
In this case,  $V$ has the unique invariant bilinear form such that  the vacuum vector $\vac$ has norm $1$, i.e., $(\vac\mymid \vac)=1$
(cf.~\cite{Li}) and the degree two subspace $V_2$ of $V$ carries a structure of a commutative
algebra called the Griess algebra of $V$.

A Virasoro vector $e\in V_2$ is called an \textit{Ising vector} if the subVOA $\la e\ra$ generated
by $e$ is isomorphic to the simple Virasoro VOA $L(\shf,0)$ of central charge $1/2$.
It is well known (cf.~\cite{DMZ,W}) that a vector $e\in V_2$ is an Ising vector
if and only if it satisfies the following relations.
\begin{equation}\label{eq:1.1}
  e_{(1)}e=2e,~~~
  4e_{(3)}e=\vac,~~~
  (64{e_{(-1)}}^3+93{e_{(-2)}}^2-264e_{(-3)}e_{(-1)}-108e_{(-5)})\vac =0,
\end{equation}
where we expand $Y(e,z)=\sum_{n\in \Z} e_{(n)}z^{-n-1}$.
If $V$ is of OZ-type, the first two relations are equivalent to saying that $e$ is
twice of an idempotent of the Griess algebra of $V$ and has norm $1/4$.
An Ising vector $e$ is said to be \textit{of $\sigma$-type on $V$} if there are no irreducible
$\la e\ra$-submodules isomorphic to $L(\shf,\sfr{1}{16})$ in $V$.
In this case, one can define an automorphism of $V$ by
\begin{equation}%\label{equation:1.2}
  \sigma_{e}
  :=
  \begin{cases}
    ~1 & \text{on }~ V[e;0],
    \\
    -1 & \text{on }~ V[e;\sfr{1}{2}],
  \end{cases}
\end{equation}
where $V[e;h]$ denotes the sum of
irreducible $\la e\ra$-submodules isomorphic to $L(\shf,h)$ for $h=0$ and $1/2$ (cf.~\cite{Mi96}). The involution $\sigma_e$ is called a $\sigma$-involution associated with $e$.
By definition, it is easy to see that 
\begin{equation}\label{conj}
  \sigma_{ge}=g\sigma_e g^{-1} \quad \text{ for any } g\in \aut\,V.
\end{equation}
Let $G_V=\la \sigma_e\mid e\in E_V\ra$, where $E_V$ is the set of Ising vectors of $V$ of $\sigma$-type. It is shown in \cite{Mi96} that if $V$ is of OZ-type then $G_V$ is a 3-transposition group (cf.~Definition \ref{df:2.1}). Under
the assumption that the VOAs have unitary forms, 3-transposition groups obtained in this manner are classified in \cite{Ma}. 
A complete list of 3-transposition groups generated by $\sigma$-involutions
as well as examples of unitary VOAs realizing these groups is obtained in \cite{Ma}; nevertheless, the classification of VOAs of OZ-type generated by Ising vectors of
$\sigma$-type is not completed yet.
\begin{comment}
It is expected that if $V$ is a VOA of OZ-type and generated by Ising vectors of $\sigma$-type
then the associated 3-transposition group $G_V$ is finite and isomorphic to one of the
examples in Matsuo's classification \cite{Ma}; furthermore, $V$ is simple and its isomorphism class is uniquely determined by the group $G_V$.
\end{comment}

It is proved in \cite{JLY} without assuming the unitarity that if an OZ-type VOA $V$ is  simple and generated by Ising vectors of $\sigma$-type, the VOA structure of $V$  is uniquely
determined by its Griess algebra which is described by the group $G_V$  (see Theorem \ref{thm:3.7}).
It is also proved in \cite{JLY} 
that if a VOA $V$ of OZ-type is generated by Ising vectors of $\sigma$-type and the associated
3-transposition group $G_V$ is isomorphic to the symmetric group $\mathfrak{S}_{n+1}$ of degree
$n+1$ for $n\geq 2$ then $V$ is simple and isomorphic to the abelian coset algebra $K(A_n,2)$
of type A (cf.~Definition \ref{KG}) which is extensively studied in \cite{JL,LS}.
The simplicity for the other cases are not established so far and the uniqueness
is still open.

In this article, we will classify vertex operator algebras of OZ-type generated by Ising vectors
of $\sigma$-type without assuming the unitarity.
It turns out that all such VOAs are unitary so that our classification coincides
with the one in \cite{Ma} (cf.~Conjecture 3.18 of \cite{JLY}).
The complete list of vertex operator algebras of OZ-type generated by Ising vectors of
$\sigma$-type is as follows.
For the definition of the abelian coset algebras $K(R,k)$ and the notion of lattice VOAs and their commutants,  we refer the readers to
Section \ref{sec:3.2}. We also refer to \cite{Ha1,Ha2} for the notation of finite groups. 

\begin{mthm}\label{thm:main}
  Let $V$ be a vertex operator algebra of OZ-type generated by Ising vectors of $\sigma$-type.
Let $E_V$ be the set of Ising vectors of $V$ of $\sigma$-type and
  $G_V=\la \sigma_e\mid e\in E_V\ra$.  
  Then $G_V$ is finite and the  
  list of the indecomposable pairs $(V,G_V)$ of a VOA $V$ and its group $G_V$ is as follows.
  \begin{equation}
  \begin{array}{c}
    (K(A_1,2),\id),~~
    (K(A_n,2),\mathfrak{S}_{n+1})~(n\geq 2),~~
    (V_{\sqrt{2}A_n}^+,F_{n+1}{:}\mathfrak{S}_{n+1})~(n\geq 2),
    \medskip\\
    (V_{\sqrt{2}D_n}^+,F_{n}^2{:}\mathfrak{S}_{n})~(n\geq 4),~~
    (K(E_6,2),\mathrm{O}_6^-(2)),~~
    (K(E_7,2),\mathrm{Sp}_6(2)),~~
    \medskip\\
    (K(E_8,2),\mathrm{Sp}_8(2)),~~
    (V_{\sqrt{2}E_6}^+,2^6{:}\mathrm{O}_6^-(2)),~~
    (V_{\sqrt{2}E_7}^+,2^6{:}\mathrm{Sp}_6(2)),~~
    \medskip\\
    (\com(K(A_2,2),V^+_{\sqrt{2}E_8}), \mathrm{O}_8^-(2)),~~
    (V_{\sqrt{2}E_8}^+,\mathrm{O}_{10}^+(2)).
  \end{array}
  \end{equation}
 The $\R$-subalgebra of $V$ generated by $E_V$ is a compact real form of $V$; thus  
  $V$ is unitary.
  
Here the natural module $F_n$ for $\fs_{n}$ is defined as $2^{n-1}$ if $n$ is odd
  and as $2^{n-2}$ if $n$ is even. The pair $(V,G_V)$ is indecomposable if $\{ \sigma_e\mid e\in E_V\}$ forms a single conjugacy class in $G_V$. 
\end{mthm}

Let us explain the outline of the proof of Theorem \ref{thm:main}.  
%For brevity, we say $V$ satisfies Condition \ref{cond:1} if $V$ is a VOA of OZ-type and generated by Ising vectors of $\sigma$-type. For a VOA $V$ satisfying Condition \ref{cond:1}, we use $E_V$ to denote  and consider the associated 3-transposition group $G_V=\la \sigma_e\mid e\in E_V\ra$. In order to prove Theorem \ref{thm:main},
First, it suffices to consider the case when the set of Ising vectors of $V$ of $\sigma$-type $E_V$ is indecomposable in the sense that
it has no non-trivial orthogonal decomposition, that is, if $E_V=E_1\perp E_2$ then either
$E_1=\varnothing$ or $E_2=\varnothing$ (see Remark \ref{indec}). We also assume $E_V$ is non-trivial, i.e., $\abs{E_V}>1$.
In this case,  we have a bijection between $e\in E_V$ and $\sigma_e\in G_V$ 
(cf.~(2) of Theorem \ref{thm:3.7}) and we can identify them.
Since $\dim V_2< \infty$, there is a finite subset $\mathcal{E}$ of $E_V$ which is indecomposable
such that $V=\la \mathcal{E}\ra$, $V_2=\C \mathcal{E}$ and $\sigma_e\mathcal{E}=\mathcal{E}$
for any $e\in \mathcal{E}$ (cf.~Lemma \ref{lem:3.6}). 
The group $G=G_{\mathcal{E}}=\la \sigma_e\mid e\in \mathcal{E}\ra$ is a finite 3-transposition subgroup of $G_V$.
As generators of $V$, we will argue about the pair $(\mathcal{E},G)$ of finite sets
rather than $(E_V,G_V)$ since we cannot assume that $E_V$ and $G_V$ are finite sets
without the unitarity.

For a 3-transposition group $(G,I)$, one can consider a commutative but non-associative
algebra $B(G)=\oplus_{i\in I}\C x^i$,
which is linearly spanned by idempotents $x^i/2$, $i\in I$ (cf.~Definition \ref{df:2.4} and \cite{Ma}). 

For $I=\{\sigma_e\in G \mid e\in \mathcal{E}\}$ and $G=\la \sigma_e\mid e\in \mathcal{E}\ra$, there is a natural epimorphism from $B(G)$ to the Griess algebra $V_2=\C \mathcal{E}$.
If $V$ is simple then the kernel of this map coincides with the radical of the bilinear form
on the algebra $B(G)$. 
The quotient of $B(G)$ by its radical is called the non-degenerate quotient.
The framework of our arguments roughly consists of the following steps.

\textbf{Step 0:}
Take a finite and indecomposable subset $\mathcal{E}$ of $E_V$ such that $V=\la \mathcal{E}\ra$,
$V_2=\C \mathcal{E}$ and $\sigma_e \mathcal{E}=\mathcal{E}$ for $e\in \mathcal{E}$.
Then the Griess algebra of $V$ is covered by the algebra $B(G)$
associated with a finite 3-transposition group $G=\la \sigma_e \mid e\in \mathcal{E}\ra$.
The expected candidates of $G$ are listed in Theorem \ref{thm:main}.

\textbf{Step 1:}
Let $G$ be one of the groups listed in Theorem \ref{thm:main}.  We will prove that the Griess algebra of $V$ is indeed isomorphic
to the non-degenerate quotient of $B(G)$ and $V$ is simple; hence the structure of $V$ is uniquely determined by \cite{JLY}.

%\textbf{Step 2:} If $G$ is one of the groups listed in Theorem \ref{thm:main},  then
%we will prove that the Griess algebra of $V$ is indeed isomorphic
%to the non-degenerate quotient of $B(G)$. It follows
%that $V$ is simple and uniquely determined by its Griess algebra by Step 1.

\textbf{Step 2:} Based on the results
obtained in Step 1,  we will show that any group $G$ not listed in Theorem \ref{thm:main} is indeed impossible.
This step completes the proof of Theorem \ref{thm:main}.

There are three infinite families in the list of groups in Theorem \ref{thm:main}; 
the most difficult case is $G=F_n{:}\mathfrak{S}_n$ with $n\geq 4$ which is isomorphic to
the quotient of the Weyl group of type $D_n$ factored by its center.
In this case, the algebra $B(F_n{:}\mathfrak{S}_n)$ is non-degenerate (cf.~Lemma \ref{lem:6.1})   and is isomorphic to the Griess algebra of $V_{\sqrt{2}A_{n-1}}^+$.  
We will prove that $V\cong V_{\sqrt{2}A_{n-1}}^+$ by induction on $n$. The main idea is similar to \cite[Theorem 4.1]{JLY} but it is much more technical since the simple Virasoro VOA of central charge $1$ is not rational. 
After establishing this case, the remaining part can be proved case-by-case, almost recursively.

The organization of this article is as follows.
In Section 2,  we review some known results on vertex operator algebras generated by Ising
vectors of $\sigma$-type. We also recall the notion of 3-transposition groups and discuss some  algebras associated with 3-transposition groups.
Some explicit examples of VOAs and 3-transposition groups are discussed.  
We will also prove an inequality in Lemma \ref{lem:3.14} which strongly constrains VOAs
satisfying Condition \ref{cond:1} and enables us to perform some recursive arguments in latter sections.
Section 3 is devoted to the proof of the isomorphism $V\cong V_{\sqrt{2}A_{n-1}}^+$ when
$G=F_n{:}\mathfrak{S}_n$ which is a part of Step 1.
In Section 4,  we will prove that if the Griess algebra of $V$ is isomorphic to the non-degenerate quotient of the algebra $B(G)$ and $V$ is simple, which completes Step 1.
The inequality in Lemma \ref{lem:3.14} will be used frequently in Step 1.
Once Step 1 is accomplished, we can easily establish Step 2 by considering Gram matrices of
the algebras $B(G)$.

\paragraph{Acknowledgement.}
H.Y.~thanks Toshiyuki Abe, Atsushi Matsuo and Kenichiro Tanabe for valuable information
through personal interactions.

\section{VOAs generated by Ising vectors of $\sigma$-type}\label{S:2}

In this paper, we consider a VOA $V$ satisfying the following condition.

\begin{cond}\label{cond:1}
  $V$ is of OZ-type and generated by its Ising vectors of $\sigma$-type.
\end{cond}

\begin{nota}\label{not:EV}
For a VOA $V$ satisfying Condition 1, we use $E_V$ to denote the set of all Ising vectors of $V$ of $\sigma$-type and set $G_V=\la \sigma_e \mymid e\in E_V\ra$.
\end{nota}

Next lemma follows from \cite[Lemma 3.7]{LSY}. 

\begin{lem}\label{lem:3.3}
If $V$ satisfies Condition \ref{cond:1}, every Ising vector of $V$ is of $\sigma$-type. 
   \end{lem}

Therefore, $E_V$ is indeed the set of all Ising vectors of $V$.
The local structures among Ising vectors are described as follows.

\begin{prop}[\cite{Mi96}]\label{prop:3.4}
  Let $V$ be a VOA of OZ-type and let $e$ and $f$ be distinct elements in $E_V$.
  Then the following hold.
  \\
  \textup{(1)} $(e\mymid f)= 0$ or $2^{-5}$.
  \\
  \textup{(2)} If $(e\mymid f)=0$ then $e_{(1)}f=0$, and if $(e\mymid f)=2^{-5}$ then
  $e_{(1)}f=\dfr{1}{4}(e+f-\sigma_e f)$ in the Griess algebra.
  In the latter case, one also has $\sigma_e f=\sigma_f e$.
  \\
  \textup{(3)} $[\sigma_e,\sigma_f]=1$ if $(e\mymid f)=0$ and $\abs{\sigma_e\sigma_f}=3$ if
  $(e\mymid f)=2^{-5}$.
%  In particular, $G_V$ is a 3-transposition group.
\end{prop}

\begin{df}[\cite{CH}]\label{df:2.1}
  A \emph{$3$-transposition group} is a pair $(G,I)$ of a group $G$ and a set of involutions $I$ of $G$  such that (1)~$G$ is generated by $I$;  (2)~$I$ is closed under the conjugation, i.e., if $a, b \in I$, $a^b=bab\in I$;  (3)~ the order of $ab$ is $\leq 3$ for any $a,b\in I$.
\end{df}

Therefore $(G_V,I_V)$, $I_V=\{ \sigma_e \mid e\in E_V\}$, is a $3$-transposition group by \eqref{conj} and Proposition \ref{prop:3.4}. 

For any  $3$-transposition group $(G,I)$, one can define a graph structure on $I$ (often called a diagram \cite[Section 3.2]{CH}) by the adjacency relation $a\sim b$ if and only if the order of $ab$ is $3$. In this case,  we have $a^b=b^a\in I$ and we set $a\circ b:=a^b=b^a$. 
Notice that $I$ is a connected graph if and only if $I$ is a single conjugacy class of $G$. We say $(G,I)$ is \emph{indecomposable or connected} if $I$ is a conjugacy class of $G$.
An indecomposable $(G,I)$ is called \emph{non-trivial} if $I$ is not a singleton, i.e.,
$G$ is not cyclic.

Using the graph structure and the relations in Proposition \ref{prop:3.4},  one can consider a commutative but non-associative algebra associated with $(G,I)$ as in \cite{Ma}. 

\begin{df} \label{df:2.4}
Let $B(G,I)=\oplus_{i\in I} \C x^i$ be the vector space spanned by a formal basis
  $\{ x^i\mid i\in I\}$ indexed by the set of involutions.
  We define a bilinear product and a bilinear form on $B(G,I)$ by
  \begin{equation}\label{eq:2.1}
    x^i\cdot x^j:=
    \begin{cases}
      2x^i & \mbox{if $i=j$},
      \\
      \dfr{1}{4} (x^i+x^j-x^{i\circ j}) & \mbox{if $i\sim j$},
      \\
      0 & \mbox{otherwise},
    \end{cases}
    ~~~~~~
    (x^i|x^j):=
    \begin{cases}
      ~\dfr{1}{4} & \mbox{if $i=j$},
      \medskip\\
      \dfr{1}{32} & \mbox{if $i \sim j$},
      \medskip\\
      ~\,0 & \mbox{otherwise}.
    \end{cases}
  \end{equation}
The bilinear form is an invariant form in the sense that $(a\cdot b| c)=(a|b\cdot c)$ for any $a,b,c\in B(G,I)$.  We also denote $B(G,I)$ by $B(G)$. 
The radical of the bilinear form on $B(G)$ forms an ideal. We
 call the quotient algebra of $B(G)$ by the radical of  the bilinear form the \emph{non-degenerate quotient}.
\end{df}

The group $G$ acts on $B(G)$; namely, there is a group homomorphism $\rho: G \to \aut\, B(G)$ 
such that 
\begin{equation}\label{eq:2.2}
 \rho(i)(x^j) = x^{i\circ j} \text{ if } i \sim j; \quad  
  \rho(i) (x^j)=x^j \text{ otherwise 
  for } i, j\in I.
\end{equation}
It is known that $\rho$ is injective if and only if $G$ is non-trivial and center-free
(cf.~\cite{Ma}).

\begin{df}\label{valency}
Suppose $G$ is indecomposable.
The \emph{valency} of the graph on $I$ is given by
$\abs{\{ j\in I \mid j\sim i\}}$  if it is finite, which is independent of the choice of $i\in I$. We denote this number by $k$ or $k_G$.
\end{df}
By \eqref{eq:2.1}, one can verify that
\[
  \l(\dsum_{i\in I} x^i\r) \cd x^j=\l(\dfr{k}4+2\r) x^j; 
\]
thus, 
\begin{equation}\label{eq:2.3}
  \w:=\dfr{8}{k+8}\dsum_{i\in I} x^i
\end{equation}
satisfies $\w\cdot  v=2v$ for any $v\in B(G)$ and $\w$ is twice the unit element  of
$B(G)$.
By the invariance, one also has $(x^i|\w)=(x^i|x^i)$ and $(\w|\w)=\dfr{2\abs{I}}{k+8}$.
\medskip

The next lemma is shown in Proposition 3.10 of \cite{JLY}.

\begin{lem}\label{lem:3.5}
  Let $V$ be a VOA satisfying Condition \ref{cond:1}. 
  Then $V_2$ is linearly spanned by $E_V$, i.e., $V_2=\C E_V$. 
\end{lem}

As a consequence, it is easy to see from the definition that 
\begin{lem}
There is a natural surjection from $B(G_V)$ to  the Griess algebra $V_2=\C E_V$, which preserves the product and the bilinear form.   
\end{lem}

\begin{df}\label{symplectic}
A subgroup $H$ of $G$ is called an $I$-subgroup if $H$ is generated by a subset of $I$.
%In this case,  $(H,H\cap I)$ is again a 3-transposition group.
A 3-transposition group $G$ is \emph{of symplectic type} if
for any two $I$-subgroups $H$ and $K$ such that $H\cong K\cong \mathfrak{S}_3$ and
$H\cap K \ne 1$, one has either $H=K$ or $\la H, K\ra\cong \mathfrak{S}_4$.
\end{df}

 3-transposition groups with trivial center are classified in \cite{CH} without the assumption of finiteness.
As a consequence of their classification, we have the following result.

\begin{thm}[\cite{CH}]\label{thm:lf}
  Every 3-transposition group is locally finite, that is, every finite subset of the
  group generates a finite subgroup.
\end{thm}

Finite 3-transposition groups of symplectic type were classified by J.I. Hall.

\begin{thm}[\cite{Ha1,Ha2}]\label{thm:hall}
An indecomposable finite center-free 3-transposition group of symplectic type is isomorphic to an extension of one of the groups
$\mathfrak{S}_3$,
$\mathfrak{S}_n$ $(n\geq 5)$,
$\mathrm{Sp}_{2n}(2)$ $(n\geq 3)$,
$\mathrm{O}_{2n}^+(2)$ $(n\geq 4)$
and $\mathrm{O}_{2n}^{-}(2)$ $(n\geq 3)$ by the direct sum of copies of the natural module.
\end{thm}
Here the natural module is isomorphic to $2^{2n}$ for $\mathrm{O}_{2n}^{\pm}(2)$ or
$\mathrm{Sp}_{2n}(2)$.
By permutations of the coordinates,
$\mathfrak{S}_n$ canonically acts on the space $\mathbb{F}_2^{n}$
where $\mathbb{F}_2$ is the field of two elements.
This action induces an embedding of $\mathfrak{S}_n$ into $\mathrm{Sp}_{n-1}(2)$
if $n$ is odd and into $\mathrm{Sp}_{n-2}(2)$ if $n$ is even,
so we define the natural module $F_n$ for $\mathfrak{S}_n$ to be $2^{n-1}$ if $n$ is odd
and %to be
$2^{n-2}$ if $n$ is even.
Note that $\mathfrak{S}_4\cong 2^2{:}\mathfrak{S}_3= F_3{:}\mathfrak{S}_3$ so that
$\mathfrak{S}_4$ is also included in the theorem above.

\medskip

Next we consider some finite generating subset of $E_V$.  
Let  $\mathcal{E}$ be a subset of $E_V$ such that $\sigma_e \mathcal{E} = \mathcal{E}$ for all $e\in \mathcal{E}$. Then $G_\mathcal{E}=\la \sigma_e \mid e\in \mathcal{E}\ra$
is a 3-transposition subgroup of $G_V$ and
the linear span $\C \mathcal{E}$ forms a subalgebra of the Griess algebra of $V$
which is a homomorphic image of $B(G_\mathcal{E})$. Such a subset $\mathcal{E}$ is called a $\sigma$-closed subset of $E_V$.

Since $V_2=\C E_V$, we can take a linear basis $A\subset E_V$ of $V_2$.
Clearly, $A$ is a finite set. 
Then $G_A=\la \sigma_e\mid e\in A\ra$ is a finite 3-transposition subgroup of $G_V$ by Theorem \ref{thm:lf}.
Set $\mathcal{E}=G_A.A=\{ ga \in E_V \mid g\in G_A,~ a\in A\}$.
Then $\mathcal{E}$ is a finite $\sigma$-closed set, $E_V\subset \C \mathcal{E}$
and $G_{\mathcal{E}}=G_A$ by \eqref{conj}.
Therefore, we have the following lemma.

\begin{lem}\label{lem:3.6}
  If $V$ satisfies Condition \ref{cond:1} then there exists a finite $\sigma$-closed
  subset $\mathcal{E}$ of $E_V$ such that $V=\la \mathcal{E}\ra$ and the Griess algebra of
  $V$ is linearly spanned by $\mathcal{E}$ which is isomorphic to a homomorphic image of $B(G_\mathcal{E})$ associated with a finite 3-transposition
  group $G_\mathcal{E}=\la \sigma_e \mid e\in \mathcal{E}\ra$. Such a subset $\mathcal{E}$ is called a \emph{finite $\sigma$-closed generating set}. 
\end{lem}

\begin{rem} \label{indec}
Let $E_V= E_1\perp E_2$ be an orthogonal partition. Then both $E_1$ and $E_2$ are $\sigma$-closed subsets of $E_V$ and we have decompositions
$V=\la E_V\ra\cong \la E_1\ra\tensor \la E_2\ra$ and $G_V=G_1\times G_2$ where
$G_i=\la \sigma_e \mid E_i\ra$, $i=1,2$.
Therefore, in order to classify VOAs satisfying Condition \ref{cond:1}, we may assume that
$E_V$ is indecomposable, that is, any orthogonal decomposition of $E_V$ is trivial.
\end{rem}

Next we recall several known results for VOAs satisfying Condition \ref{cond:1}.

\begin{thm}[\cite{Ma,JL,JLY}]\label{thm:3.7}
  Let $V$ be a VOA satisfying Condition \ref{cond:1}.
  \\
  \textup{(1)}~$G_V=\la \sigma_e \mid e\in E_V\ra$ is a 3-transposition group of symplectic type.
  \\
  \textup{(2)}~If $G_V$ is connected and non-trivial then the map
  $E_V\ni e\mapsto \sigma_e \in G_V$ is injective.
  \\
  \textup{(3)}~The Griess algebra of $V$ is linearly spanned by $E_V$ and
  is isomorphic to a homomorphic image of the algebra $B(G_V)$.
  \\
  \textup{(4)}~If $V$ is simple then the Griess algebra of $V$ is isomorphic to
  the non-degenerate quotient of the algebra $B(G_V)$ and the VOA structure of $V$ is uniquely determined by its Griess algebra.
\end{thm}

The statement (1) of Theorem \ref{thm:3.7}  follows from Proposition 1 of \cite{Ma}
and Proposition 3.15 of \cite{JLY},
(3) follows from Lemma 3.5 of \cite{JL} and Proposition 3.9 of \cite{JLY}
and (4) follows from Theorem 3.13 of \cite{JLY}.

\medskip

The following result is also proved in \cite{Ma}. 
\begin{lem}[\cite{Ma}]\label{lem:6.1}
  Let $(G,I)$ be a finite center-free indecomposable 3-transposition group of symplectic type
  with $\abs{I}>1$ and $B(G)=\oplus_{i\in I}\C x^i$ the algebra associated with $(G,I)$.
  Then the bilinear form on the real part $\oplus_{i\in I}\R x^i$ of $B(G)$ is
  positive semidefinite if and only if $G$ is one of those listed in Theorem \ref{thm:main}.
  More precisely, the bilinear form is positive definite if and only if $G$ is one of the following.
  \begin{equation}\label{eq:6.1}
    \mathfrak{S}_n~(n\geq 3),~~~
    F_n{:}\mathfrak{S}_{n}~(n\geq 4),~~~
    \mathrm{O}_{6}^-(2), ~~~
    \mathrm{Sp}_{6}(2), ~~~
    \mathrm{O}_{8}^+(2),
  \end{equation}
  where $F_n$ denotes the natural module for $\mathfrak{S}_n$.
  It is positive semidefinite and singular if and only if $G$ is one of the following.
  \begin{equation}\label{eq:6.2}
    F^2_n{:}\mathfrak{S}_n~(n\geq 4),~~
    2^6{:}\mathrm{O}_6^-(2),~~
    2^6{:}\mathrm{Sp}_6(2),~~
    \mathrm{O}_8^-(2),~~
    \mathrm{Sp}_8(2),~~
    2^8{:}\mathrm{O}_8^+(2),~~
    \mathrm{O}_{10}^+(2).
  \end{equation}
\end{lem}

By this lemma, if $G$ is one of the groups in \eqref{eq:6.1}, $B(G)$ is non-degenerate and isomorphic to its non-degenerate quotient. 
%For the discussions in the next sections, we introduce the following notion.
%By Lemma \ref{lem:3.6}, there exists a finite $\sigma$-closed subset $\mathcal{E}$  of $E_V$ such that $V=\la \mathcal{E}\ra$
%and $V_2=\C \mathcal{E}$.
\medskip

Let $V$ be a VOA satisfying Condition \ref{cond:1} and $\mathcal{E}\subset E_V$ a finite $\sigma$-closed generating set. 

\begin{df}\label{df:sigma}
 A module $M$ over $V$ is said to be of \emph{$\sigma$-type} with respect to $\mathcal{E}$
  if any irreducible $\la e\ra$-submodule of $M$ is isomorphic to either $L(\shf,0)$ or
  $L(\shf,\shf)$ for every $e\in \mathcal{E}$.
\end{df}

%Set $G=\la \sigma_e \mid e\in \mathcal{E}\ra$.
%Then $G$ is a finite 3-transposition group by Proposition \ref{prop:3.4}.
%Suppose $G$ is indecomposable.

\begin{lem}\label{lem:3.13} %Let $V$ be a VOA satisfying Condition \ref{cond:1}.
  Let $M=\oplus_{i\geq 0} M(i)$ be an irreducible $\N$-gradable module over $V$ of $\sigma$-type with respect to a finite
  $\sigma$-closed subset $\mathcal{E}$.
  Suppose $G=\la \sigma_e\mid e\in \mathcal{E}\ra$ is indecomposable.
  Then for any $e$, $f\in \mathcal{E}$, we have
  \[
    \mathrm{Tr}_{M(i)}e_{(1)} = \mathrm{Tr}_{M(i)}f_{(1)}.
  \]
\end{lem}

\pf
Since $G$ is indecomposable, all the Ising vectors in $\mathcal{E}$ are mutually conjugate under $G$.
Since $M$ is of $\sigma$-type, it is $g$-stable for every $g\in G$.
Then the claim follows from the invariance of traces under conjugation.
\qed
\medskip

Suppose $G$ is indecomposable.
Denote $I_G=\{ \sigma_e \mid e\in \mathcal{E}\}$.
It follows from (2) of Theorem \ref{thm:3.7} that $\abs{I_G}=\abs{\mathcal{E}}$.
Let $H$ be a non-trivial indecomposable 3-transposition subgroup of $G$
where the set of 3-transpositions of $H$ is given by $I_H=H\cap I_G$.
Recall that $I_G$ and $I_H$ have graph structures defined as in Section \ref{S:2}.
We denote the valencies of $I_G$ and $I_H$ by $k_G$ and $k_H$, respectively (see Definition \ref{valency}).
Let $\mathcal{E}'=\{ e\in \mathcal{E} \mid \sigma_e \in H\}$ and consider the subVOA
$W=\la \mathcal{E}'\ra$ of $V$.
Let $\w_V$ and $\w_W$ be the conformal vectors of $V$ and $W$, respectively.
We will frequently make use of the following inequality in the sequel.

\begin{lem}\label{lem:3.14}
  Let $J$ be a $V$-module of $\sigma$-type with respect to $\mathcal{E}$ with a real top weight $h$.
  Suppose that $J$ is a direct sum of irreducible $W$-submodules and there exists a bound
  $\lambda\in \R_{>0}$ such that top weights of irreducible $W$-submodules of $J$ are all real
  and less than or equal to $\lambda$.
  Then the following inequality holds.
  \[
    h \leq \dfr{\lambda(k_H+8)\abs{I_G}}{(k_G+8)\abs{I_H}}.
  \]
\end{lem}

\pf
%Since $\sigma_e=(-1)^{2e_{(1)}}$ for $e\in E_V$, any $V$-module of $\sigma$-type is $G_V$-invariant.
By \eqref{eq:2.3},  we have
\begin{equation}\label{eq:5.1}
  \tr_{J_h} (\w_V)_{(1)}= \dfr{8\abs{I_G}}{k_G+8} \tr_{J_h} {x^i}_{(1)}=h\dim J_h
\end{equation}
for any $i\in I_G$.
On the other hand, let $\lambda$ be a bound as in the statement.
Then again by \eqref{eq:2.3} we have
\begin{equation}\label{eq:5.2}
  \tr_{J_h} (\w_W)_{(1)}= \dfr{8\abs{I_H}}{k_H+8} \tr_{J_h} {x^i}_{(1)} \leq \lambda \dim J_h
\end{equation}
for any $i\in I_H$.
Since $J$ is of $\sigma$-type, we can apply Lemma \ref{lem:3.13} and obtain the inequality
from \eqref{eq:5.1} and \eqref{eq:5.2}.
\qed

\subsection{Examples of VOAs satisfying Condition \ref{cond:1}}\label{sec:3.2}
Next we discuss some explicit examples and review some of their properties.

\paragraph{Scaled root lattices. }
Let $R$ be an irreducible root lattice. We use $\Delta_R$ (or $\Delta$) to denote  its root system, i.e.,
the set of vectors in $R$ with squared norm 2.
Let $\Delta_R=\Delta_R^+\sqcup (-\Delta_R^+)$ be a partition of $\Delta_R$ into
positive and negative roots.
Denote by $l$ and $h$ the rank and the Coxeter number of $R$, respectively.
We denote by $\sqrt{2}R$ the lattice whose norm is twice of $R$.
Let $V_{\sqrt{2}R}^+$ be the fixed point subalgebra of the lattice VOA $V_{\sqrt{2}R}$
under the lift of $(-1)$-isometry on $R$.
Then $V_{\sqrt{2}R}^+$ is of OZ-type.
For a root $\alpha \in \Delta_R$, set
\begin{equation}\label{eq:wpm}
%  w^\pm(\alpha) := \dfr{1}{16}\alpha_{(-1)}^2\vac \pm \dfr{1}{4}(e^\alpha+e^{-\alpha})
   w^\pm(\alpha) := \dfr{1}{8}\alpha_{(-1)}^2\vac \pm \dfr{1}{4}(e^{\sqrt{2}\alpha}+e^{-\sqrt{2}\alpha})
  \in V_{\sqrt{2}R}^+.
\end{equation}
Then $w^\pm(\alpha)$ are Ising vectors of $V_{\sqrt{2}R}^+$ (cf.~\cite{DMZ}) and $V_{\sqrt{2}R}^+$ is
generated by Ising vectors $\{ w^\pm(\alpha) \mid \alpha\in \Delta_R^+\}$
(cf.~ \cite[Proposition 12.2.6]{FLM}).
Moreover, all Ising vectors of $V_{\sqrt{2}R}^+$ are of $\sigma$-type (cf.~\cite[Corollary 4.5]{LSY}).  
Thus $V_{\sqrt{2}R}^+$ satisfies Condition \ref{cond:1}.
\begin{lem}\label{lem:prod}
The inner products between Ising vectors $\{ w^\pm(\alpha) \mid \alpha\in \Delta_R^+\}$ are as follows.
\begin{equation}\label{eq:inner}
  (w^{\epsilon}(\alpha)\mymid w^{\epsilon'}(\beta))
  = \begin{cases}
    \dfr{1}{4} & \mbox{if $\alpha=\pm \beta$ and $\epsilon=\epsilon'$},
    \medskip\\
    0 & \mbox{if ``$(\alpha\mymid \beta)=0$'' or
    ``$\alpha=\pm \beta$ and $\epsilon=-\epsilon'$\,''},
    \medskip\\
    \dfr{1}{32} & \mbox{if $(\alpha\mymid \beta)=\pm 1$}.
  \end{cases}
\end{equation}
One also has the following product in the Griess algebra of $V_{\sqrt{2}R}^+$.
\begin{equation}\label{eq:prod}
    w^{\epsilon}(\alpha)_{(1)}w^{\epsilon'}(\beta)
  = \dfr{1}{4}\l( w^{\epsilon}(\alpha)+w^{\epsilon'}(\beta)-w^{-\epsilon\epsilon'}(r_\alpha(\beta))\r) ,  \quad \text{ if } (\alpha\mymid \beta)=\pm 1 
\end{equation}
where $r_\alpha\in O(R)$ is the reflection associated with $\alpha$.
\end{lem}

It follows from Lemma \ref{lem:prod} that the Ising vectors $w^\pm(\alpha)$,
$\alpha\in \Delta_R^+$,  span a linear subspace which is closed under the product of
the Griess algebra of $V_{\sqrt{2}R}^+$ and they span the Griess algebra (cf.~\cite{DLMN}).
Indeed, the set $\{ w^\pm(\alpha) \mid \alpha\in \Delta_R^+\}$
exhausts all Ising vectors of $V_{\sqrt{2}R}^+$ except for the case $R=E_8$
(cf.~\cite[Theorem 4.4]{LSY}).

The 3-transposition groups $G_R=\la I_R\ra$ and the associated VOAs are given in Table 1  (cf.~\cite{Ma}),
where the parameters $I_R=\{ \sigma_{w^\pm(\alpha)} \mid \alpha \in \Delta_R^+\}$,  
$k=k_{G_R}$ is defined as in Definition  \ref{valency}, $c$ is the central charge, $d$ denotes the dimension of the Griess algebra, and $F_n$ denotes the natural module for the symmetric group.

\[
\renewcommand{\arraystretch}{2}
\begin{array}{c}
\begin{array}{|c|c|c|c|c|c|}
  \hline
  V_{\sqrt{2}R}^+ & G_R & \abs{I_R} & k & ~c~ & d
  \\ \hline
  ~V_{\sqrt{2}A_n}^+~\mbox{\small ($n>1$)}~ & F_{n+1}{:}\mathfrak{S}_{n+1} & n(n+1) & 4(n-1) & n
  & n(n+1)
  \\ \hline
  ~V_{\sqrt{2}D_n}^+~\mbox{\small ($n\geq 4$)}~ & F_n^2{:}\mathfrak{S}_n & ~2n(n-1)~ & ~8(n-2)~
  & ~n~ & ~\dfr{n(3n-1)}{2}~
  \\ \hline
  V_{\sqrt{2}E_6}^+ & 2^6{:}\mathrm{O}_6^-(2) & 72 & 40 & 6 & 57
  \\ \hline
  V_{\sqrt{2}E_7}^+ & ~2^6{:}\mathrm{Sp}_6(2)~ & 126 & 64 & 7 & 91
  \\ \hline
  V_{\sqrt{2}E_8}^+ & 2^8{:}\mathrm{O}_8^+(2) & 240 & 112 & 8 & 156
%  \\ \hline
%  V_{\sqrt{2}E_8}^+ & \mathrm{O}_{10}^+(2) & 496 & 240 & 8 & 156 & \mbox{extra Isings}
  \\ \hline
\end{array}
\\
\mbox{Table 1: Parameters of Griess algebras and 3-transposition groups}
\end{array}
\renewcommand{\arraystretch}{1}
\]

\paragraph{Commutants and Abelian Coset algebras.}
By Lemma \ref{lem:prod}, the linear span of
the subset $\{ w^-(\alpha) \mid \alpha \in \Delta_R^+\}$ is also closed
under the product in the Griess algebra of $V_{\sqrt{2}R}^+$.
As in \cite{LSY}, we denote by $M_R$ the subVOA generated by
$\{ w^-(\alpha) \mid \alpha \in \Delta_R^+\}$.  Note that the Griess algebra of $M_R$ is spanned linearly by $\{ w^-(\alpha) \mid \alpha \in \Delta_R^+\}$.  
It turns out that $M_R$ is also isomorphic to the abelian coset algebra $K(R,2)$ defined as follows.

\begin{df}[\cite{DW}] \label{KG}
Let $\mathfrak{g}$ be the simple Lie algebra associated with a root system of type $R$.  
For $k\in \Z_{\geq 0}$, let $L_{\hat{\mathfrak{g}}}(k,0)$ be the simple affine VOA associated
with $\hat{\mathfrak{g}}$ at level $k$ and $M_{\mathfrak{h}}$ the Heisenberg subVOA generated
by a Cartan subalgebra $\mathfrak{h}$ of the Lie algebra  on the weight one subspace of
$L_{\hat{\mathfrak{g}}}(k,0)$ which can be identified with $\mathfrak{g}$.
The abelian coset subalgebra is defined as the commutant
\begin{equation}\label{eq:3.3}
	K(R,k)= K(\mathfrak{g},k):=\com(M_{\mathfrak{h}},L_{\hat{\mathfrak{g}}}(k,0))
\end{equation}
of $M_{\mathfrak{h}}$ in $L_{\hat{\mathfrak{g}}}(k,0)$, 
where the commutant of $U$ in $V$ is defined as 
\[
\com(U,V) =\{ v\in V\mid u_{(n)} v=0 \text{ for all } u\in U, n\in \Z_{\geq 0} \} 
\]
for a subalgebra $U$ of $V$. The central charge of $K(\mathfrak{g},k)$ is $hl(k-1)/(h+k)$, where $l$ and $h$ are
the rank and the Coxeter number of $\mathfrak{g}$, respectively.
\end{df}
\begin{comment}
It is known \cite{DW} that the commutant of $K(\mathfrak{g},k)$ in $L_{\hat{\mathfrak{g}}}(k,0)$
is isomorphic to the lattice VOA $V_{\sqrt{k}R}$ and 
$K(\mathfrak{g},k)\tensor V_{\sqrt{k}R}$ is a conformal subVOA of $L_{\hat{\mathfrak{g}}}(k,0)$.  
\end{comment}
It is shown in \cite{DW} that for each $\alpha\in \Delta_R^+$,
there is an embedding
\begin{equation}\label{eq:embed}
  \iota_{\alpha,k} :K(\mathfrak{sl}_2,k)\to  K(\mathfrak{g},k),
\end{equation}
and $K(\mathfrak{g},k)$ is generated by subalgebras
$\mathrm{Im}\,\iota_{\alpha,k}\cong K(\mathfrak{sl}_2,k)$, $\alpha\in \Delta_R^+$, as a VOA.

For $k=2$, $K(\mathfrak{sl}_2,2)$ has central charge $1/2$ and it 
is isomorphic to $L(\shf,0)$.
Let  $x^\alpha$ be 
the conformal vector of $\mathrm{Im}\,\iota_{\alpha,2}$. Then  $x^\alpha$ is an Ising vector of
$K(\mathfrak{g},2)$, and  $K(\mathfrak{g},2)$ is generated by $x^\alpha$,
$\alpha\in \Delta_R^+$.

\begin{lem}
We have  $K(R,2) \cong M_R$. 
\end{lem}
\begin{proof}
Let $\mathfrak{g}$ be the simple Lie algebra associated with the root lattice $R$.
Recall that $L_\mathfrak{g}(1,0) \cong V_R$ and the root lattice $R\oplus R$ contains two mutually orthogonal copies of $\sqrt{2} R$, namely  $\mathcal{R}^+ =\{ (\alpha,\alpha)\in R\oplus R \mid  \alpha\in R\}$ and $\mathcal{R}^- =\{ (\alpha,-\alpha)\in R\oplus R\mid \alpha\in R\}$.

One can embed  the level 2 affine VOA $L_\mathfrak{g}(2,0)$ diagonally into 
$L_\mathfrak{g}(1,0) \otimes L_\mathfrak{g}(1,0) \cong V_R\otimes V_R$, namely,
$L_\mathfrak{g}(2,0)$  is generated by  $(\alpha, \alpha)_{(-1)} \cdot \vac, e^{(\alpha,0)} + e^{(0,\alpha)}$ with $\alpha \in \Delta_R= R(2)$.   

The subspace $\mathfrak{h}=\mathrm{span}\{ (\alpha, \alpha)_{(-1)}\cdot \vac| \alpha \in \Delta_R\}$ is a Cartan subalgebra of  $L_\mathfrak{g}(2,0)_1$. The corresponding Heisenberg VOA $M_{\mathfrak{h}}$ agrees with the  Heisenberg subVOA of $V_{\mathcal{R}^+}$. It follows that  $K(R,2) \subset V_{\mathcal{R}^-}$. 

The embedding \eqref{eq:embed} is based on an embedding $L_{\hat{\mathfrak{sl}}_2}(2,0)\subset L_{\hat{\mathfrak{g}}}(2,0)$. Let $S_\alpha= \mathrm{span} 
\{ (\alpha, \alpha)_{(-1)} \cdot \vac, e^{(\alpha,0)} + e^{(0,\alpha)}, e^{-(\alpha,0)} + e^{-(0,\alpha)}\}$ be the $\mathfrak{sl}_2$-triple defined by $\alpha$.  Then $x^\alpha= \omega_{L_{\hat{S}_\alpha}(2,0)} - 1/8 (\alpha,\alpha)_{(-1)}^2\cdot \vac$.  By the Sugawara construction,  it is straightforward to show $x^\alpha= \omega^{-}(\tilde{\alpha})$, where $\tilde{\alpha} =(\alpha, -\alpha)\in \mathcal{R}^-$.  Thus, $K(R,2)\cong M_R$ since they are both generated by  $\omega^{-}(\tilde{\alpha})\in V_{\mathcal{R}^-}, \alpha \in \Delta_R^+$. 
\end{proof}

The set of all Ising vectors of $M_R\cong K(\mathfrak{g},2)$
are classified in \cite{LSY}; it is given by
$\{ x^\alpha\mid \alpha \in \Delta_R^+\}$ except for the case $\mathfrak{g}$ is of $E_8$ type.
Let $I_{\Delta_R} =\{ \sigma_{x^\alpha} \mid \alpha \in \Delta_R^+\}$. Then the 3 transposition  $G_{\Delta_R}=\la I_{\Delta_R}\ra$ for simple root lattices are as follows (cf.~\cite{Ma}), where 
$\abs{I_\Delta}=\abs{\Delta^+}$ and $k$, $c$, $d$  and $F_n$ are defined as in Table 1. 

\[
\renewcommand{\arraystretch}{2}
\begin{array}{c}
\begin{array}{|c|c|c|c|c|c|}
  \hline
  V & G_{\Delta_R} & \abs{I_{\Delta_R}} & k & c & d
  \\ \hline
  K(A_n,2) ~\mbox{\small ($n>1$)} & \mathfrak{S}_{n+1} & ~\dfr{n(n+1)}{2}~ & 2(n-1)
  & ~\dfr{n(n+1)}{n+3}~ & ~\dfr{n(n+1)}{2}~
  \\ \hline
  K(D_n,2) ~\mbox{\small ($n\geq 4$)} & F_n{:}\mathfrak{S}_n & n(n-1) & ~4(n-2)~ & n-1 & n(n-1)
  \\ \hline
  K(E_6,2) & \mathrm{O}_6^-(2) & 36 & 20 & \dfr{36}{7} & 36
  \\ \hline
  K(E_7,2) & ~\mathrm{Sp}_6(2)~ & 63 & 32 & \dfr{63}{10} & 63
  \\ \hline
  K(E_8,2) & \mathrm{O}_8^+(2) & 120 & 56 & \dfr{15}{2} & 120
%  \\ \hline
%  K(E_8,2) & \mathrm{Sp}_8(2) & 255 & 128 & \dfr{15}{2} & \mbox{extra Isings}
  \\ \hline
\end{array}
\\
\mbox{Table 2: Parameters of Griess algebras and 3-transposition groups}
\end{array}
\renewcommand{\arraystretch}{1}
\]

\paragraph{$V_{\sqrt{2}E_8}^+$ and $M_{E_8}$.}  
By \cite{DLMN} (see also \eqref{eq:2.3}),  the conformal vectors of $M_R$ and $V_{\sqrt{2}R}^+$ are given by 
\begin{equation}\label{eq:3.8}
  \eta =\dfr{4}{h+2}\dsum_{\alpha\in \Delta_R^+} w^-(\alpha)\quad \text{ and } \quad 
  %=\dfr{2}{h+2}\dsum_{\alpha\in \sqrt{2}R\atop (\alpha|\alpha)=4} w^-(\alpha),~~~~~
  \w =\dfr{1}{2h}\dsum_{\alpha\in \Delta_R^+} \alpha_{(-1)}^2\vac, 
  %=\dfr{1}{8h}\dsum_{\alpha\in \sqrt{2}R\atop (\alpha|\alpha)=4} \alpha_{(-1)}^2\vac , 
\end{equation}
respectively and the central charges of $\eta$ and $\w$ are $hl/(h+2)$ and $l$, respectively.
Then  
$ %\begin{equation}\label{eq:xi}
\xi:=\w-\eta
$ %\end{equation}
is a Virasoro vector of central charge $2l/(h+2)$ (cf.~\cite[Eq. (2.5)]{DLMN}).  
In the case of $R=E_8$, the central charge of $\xi$ is $1/2$ so that $\xi$ is an Ising vector. That means $V_{\sqrt{2}E_8}^+$ contains some Ising vectors not contained in the set $\{ w^\pm(\alpha) \mid \alpha\in \Delta^+\}$. 

For $x\in E_8$, define $\varphi_x:=(-1)^{(1/\sqrt{2})x_{(0)}}$.
Then $\varphi_x$ is an automorphism of $V_{\sqrt{2}E_8}^+$ and the map $x\mapsto \varphi_x$
defines a homomorphism from the additive group $E_8$ to $\aut\,V_{\sqrt{2}E_8}^+$ with
kernel $2E_8$.
Therefore, we have totally $[E_8:2E_8]=2^8$ Ising vectors of the form $\varphi_x (\xi)$,
$x\in E_8$, in $V_{\sqrt{2}E_8}^+$. 
We will call Ising vectors of the form $\varphi_x (\xi)$ of \emph{$E_8$-type},
whereas those in $\{ w^\pm(\alpha) \mid \alpha\in \Delta^+\}$ of \emph{$A_1$-type}.
The VOA $V_{\sqrt{2}E_8}^+$ contains 240 Ising vectors of $A_1$-type and 256 Ising vectors
of $E_8$-type and the set of Ising vectors is given by the union of those
Ising vectors (cf.~\cite{G}).
Since $M_{E_8}\cong K(E_8,2)$ is a subVOA of $V_{\sqrt{2}E_8}^+$, it also contains 135
Ising vectors of $E_8$-type other than 120 Ising vectors of $A_1$-type (cf.~\cite{LSY}).
Let  $I_V=\{ \sigma_e \mid e\in E_V\}$. The group $G_V=\la I_V\ra,$ $V=V_{\sqrt{2}E_8}^+$ or $K(E_8,2)$  are  as follows.

\[
\renewcommand{\arraystretch}{2}
\begin{array}{c}
\begin{array}{|c|c|c|c|c|c|}
  \hline
  V & G_V & \abs{I_V} & k & c & d
  \\ \hline
  ~K(E_8,2)~ & ~\mathrm{Sp}_8(2)~ & 255=120+135 & 128 & ~\dfr{15}{2}~ & 120
  \\ \hline
  V_{\sqrt{2}E_8}^+ & \mathrm{O}_{10}^+(2) & ~496=240+256~ & ~240~ & 8 & ~156~
  \\ \hline
\end{array}
\\
\mbox{Table 3: Parameters of Griess algebras and 3-transposition groups}
\end{array}
\renewcommand{\arraystretch}{1}
\]
There is one more example from Matsuo's classification, which is related to the commutant VOA $\com(K(A_2,2),V_{\sqrt{2}E_8}^+)$ \cite{LSY,Ma}. We also include the following table for completeness.

\[
\renewcommand{\arraystretch}{2}
\begin{array}{c}
\begin{array}{|c|c|c|c|c|c|}
  \hline
  V & G_V & \abs{I_V} & k & c & d
  \\ \hline
  ~\com(K(A_2,2),V_{\sqrt{2}E_8}^+)~ & ~\mathrm{O}_8^-(2)~ & ~136~ & ~72~ & ~\dfr{34}{5}~ & ~85
  \\ \hline
\end{array}
\\
\mbox{Table 4: Parameters of the Griess algebras and the 3-transposition group}
\end{array}
\renewcommand{\arraystretch}{1}
\]

\begin{rem}\label{rem:3.8}
The non-degenerate quotients of the algebras $B(G)$ associated with
$G=\mathrm{O}_{10}^+(2)$ and $2^8{:}\mathrm{O}_8^+(2)$ are isomorphic and realized by
the Griess algebra of $V_{\sqrt{2}E_8}^+$.
Similarly, $B(G)$ associated with $G=\mathrm{O}_8^+(2)$ and
$\mathrm{Sp}_8(2)$ have isomorphic non-degenerate quotients which are realized by
the Griess algebra of $K(E_8,2)$.
This implies that there are two $\sigma$-closed generating sets in each of $V_{\sqrt{2}E_8}^+$ and
$K(E_8,2)$. 
\end{rem}

\begin{rem}\label{rem:3.9}
  It is known (cf.~\cite{DLY1,LSY}) that $K(D_n,2)\cong V_{\sqrt{2}A_{n-1}}^+$ for $n\geq 3$,
  where we identify $D_3=A_3$.
  
\begin{comment}
There is an involution $\sigma$ of $V_{A_1^n}$ such that
$\sigma(w^-(\alpha_i-\alpha_j))=w^-(\alpha_i-\alpha_j)$,
$\sigma(w^+(\alpha_i-\alpha_j))=w^-(\alpha_i+\alpha_j)$,
$\sigma(w^-(\alpha_i+\alpha_j))=w^+(\alpha_i-\alpha_j)$,
$\sigma(w^+(\alpha_i+\alpha_j))=w^+(\alpha_i+\alpha_j)$
so that $\sigma$ induces an isomorphism from $M_{D_n}$ to $V_{\sqrt{2}A_{n-1}}^+$.
\end{comment}
\end{rem}

\section{Simplicity for type $D$}
Let $n\geq 4$ and let $V^{(n)}$ be a VOA satisfying Condition \ref{cond:1}. Assume that there exists a finite $\sigma$-closed generating set $\mathcal{E}$ of $V^{(n)}$ such that $G=\la \sigma_e\mid e\in \mathcal{E}\ra$ is isomorphic to $F_n{:}\mathfrak{S}_n$.
It follows from  Theorem \ref{thm:3.7} (4) that
$V^{(n)}$ has the unique simple quotient isomorphic to $K({D_{n}},2)\cong V_{\sqrt{2}A_{n-1}}^+$ (cf.~Table 1 and Remark \ref{rem:3.9}). In this section,  we will prove that $V^{(n)}$ is indeed simple and isomorphic to $V_{\sqrt{2}A_{n-1}}^+$.

Since the algebra $B(F_n{:}\mathfrak{S}_n)$ is non-degenerate (cf.~Lemma \ref{lem:6.1}), $B(F_n{:}\mathfrak{S}_n)$ is isomorphic to the Griess algebra of $V^{(n)}$ and the Griess  algebra of its simple quotient $V_{\sqrt{2}A_{n-1}}^+$. Therefore, we may identify 
the set of Ising vectors of $V^{(n)}$ with those of $V_{\sqrt{2}A_{n-1}}^+$. Recall that the Griess algebra of $V_{\sqrt{2}A_{n-1}}^+$  has a basis consisting of Ising vectors
\begin{equation}\label{eq:4.2}
  \{ w^\pm(\al) \mid \al \in \Delta^+_{A_{n-1}} \}.
\end{equation}
Without loss of generality, we may view it as a subset of $V^{(n)}$.

\medskip

Notice that $V^{(n)}$ itself is a $V^{(n)}$-module of $\sigma$-type (cf.~Definition \ref{df:sigma}); 
thus, any ideals or quotient modules of $V^{(n)}$ are $V^{(n)}$-modules of $\sigma$-type. 
In order to analyze the structure of $V^{(n)}$, we will first classify irreducible $V^{(n)}$-modules of $\sigma$-type.

\subsection{Irreducible $V^{(n)}$-modules of $\sigma$-type}
%Since $V^{(n)}$ has the unique simple quotient isomorphic to $V_{\sqrt{2}A_{n-1}}^+$
Since there is a surjection from $V^{(n)}$ to $K(D_n,2)\cong V_{\sqrt{2}A_{n-1}}^+$ 
by Theorem
\ref{thm:3.7}, irreducible $V_{\sqrt{2}A_{n-1}}^+$-modules of $\sigma$-type
are also irreducible $V^{(n)}$-modules of $\sigma$-type.
We first consider irreducible $V_{\sqrt{2}A_{n-1}}^+$-modules of $\sigma$-type.

Let $\{\alpha_1,\dots,\alpha_{n-1}\}$ be a set of simple roots of the lattice $A_{n-1}$ such that
$(\alpha_i\mymid \alpha_i)=2$ for $1\leq i<n$,
$(\alpha_i\mymid \alpha_j)=-1$ if $\abs{i-j}=1$
and $(\alpha_i\mymid \alpha_j)=0$ if $\abs{i-j}>1$ for $1\leq i<j<n$.
Let
\begin{equation}\label{eq:4.1}
\Lambda_{n-1}=\frac{\sqrt{2}}{2n}(\alpha_1+2\alpha_2+\cds +(n-1)\alpha_{n-1}). 
\end{equation} 
Note that 
$\Lambda_{n-1}$ is the unique vector in $(\sqrt{2}A_{n-1})^*$ such that
$(\Lambda_{n-1}\mymid \sqrt{2}\alpha_i)=\delta_{i,n-1}$ for $1\leq i<n$.

\begin{lem}\label{lem:4.1}
  The number of inequivalent irreducible $V_{\sqrt{2}_{A_{n-1}}}^+$-modules of $\sigma$-type
  is equal to $(n+3)/2$ if $n$ is odd and to $n/2+3$ if $n$ is even.
\end{lem}

\pf
It is shown in Theorem 4.4 of \cite{LSY} that for any Ising vector $x$ of
$V_{\sqrt{2}A_{n-1}}^+$, there exists a vector $\beta \in \sqrt{2}A_{n-1}$ of squared norm 4
such that $x\in V_{\Z\beta}^+\subset V_{\sqrt{2}A_{n-1}}^+$.
The irreducible modules of $V_{\Z\beta}^+$ and $V_{\sqrt{2}A_{n-1}}^+$ are classified in \cite{AD}
and it turns out that any irreducible $V_{\sqrt{2}A_{n-1}}^+$-module is a submodule of
either untwisted or $\Z_2$-twisted irreducible $V_{\sqrt{2}A_{n-1}}$-modules.
It is directly verified that $x$ is not of $\sigma$-type on irreducible
$V_{\Z\beta}^+$-submodules contained in a twisted $V_{\Z\beta}$-module, and $x$ is of
$\sigma$-type on a $V_{\Z\beta}^+$-module of untwisted type if and only if it is a submodule
of either $V_{\Z\beta}$ or $V_{\Z\beta+\beta/2}$.
Therefore, if $M$ is an irreducible $V_{\sqrt{2}A_{n-1}}^+$-module of $\sigma$-type,
then it is untwisted type and there is a $\lambda\in (\sqrt{2}A_{n-1})^*$ such that
$M\subset V_{\sqrt{2}A_{n-1}+\lambda}$ and $(\lambda\mymid \sqrt{2}A_{n-1})\subset 2\Z$.
The cosets $\lambda+\sqrt{2}A_{n-1}$ satisfying the latter condition are given by
$2i\Lambda_{n-1}+\sqrt{2}A_{n-1}$, $0\leq i<n$ (cf.~Lemma 4.1 of \cite{AYY}).
It is known that $V_{\sqrt{2}A_{n-1}+\lambda}\cong V_{\sqrt{2}A_{n-1}-\lambda}$ as
$V_{\sqrt{2}A_{n-1}}^+$-modules for any $\lambda\in (\sqrt{2}A_{n-1})^*$ by Proposition 3.1
of \cite{AD}.
If $n$ is odd, all $V_{\sqrt{2}A_{n-1}+2i\Lambda_{n-1}}$, $1\leq i\leq (n-1)/2$, are inequivalent
irreducible $V_{\sqrt{2}A_{n-1}}^+$-modules whereas in the case $i=0$ it splits into a direct sum
$V_{\sqrt{2}A_{n-1}}=V_{\sqrt{2}A_{n-1}}^+\oplus V_{\sqrt{2}A_{n-1}}^-$ of two inequivalent
irreducible $V_{\sqrt{2}A_{n-1}}^+$-modules (cf.~\cite{AD}).
Therefore, the number of inequivalent irreducible $V_{\sqrt{2}A_{n-1}}^+$-modules of
$\sigma$-type is $(n-1)/2+2=(n+3)/2$ as claimed.

If $n$ is even, by \cite{AD},
all $V_{\sqrt{2}A_{n-1}+2i\Lambda_{n-1}}$, $1\leq i\leq n/2-1$, are inequivalent
irreducible $V_{\sqrt{2}A_{n-1}}^+$-modules of $\sigma$-type whereas we have the following splitting
in the case $i=0$ and $i=n/2$:
\[
  V_{\sqrt{2}A_{n-1}} = V_{\sqrt{2}A_{n-1}}^+\oplus V_{\sqrt{2}A_{n-1}}^-,~~~
  V_{\sqrt{2}A_{n-1}+n\Lambda_{n-1}} = V_{\sqrt{2}A_{n-1}+n\Lambda_{n-1}}^+ \oplus V_{\sqrt{2}A_{n-1}+n\Lambda_{n-1}}^-.
\]
Therefore, the number of inequivalent irreducible $V_{\sqrt{2}A_{n-1}}^+$-modules of
$\sigma$-type is $n/2-1+4=n/2+3$ as in the assertion.
\qed

\begin{rem}\label{rem:4.2}
  It follows from Theorem A.1 of \cite{AYY} that
  the top weight and the dimension of the top level of $V_{\sqrt{2}A_{n-1}+2i\Lambda_{n-1}}$
  are $i(n-i)/n$ and $\binom{n}{i}$, respectively, for $0\leq i<n$.

Recall that for a VOA $V$ and an $\N$-gradable $V$-module $M=\oplus_{i\geq 0} M(i)$ with $M(0)\neq 0$, the lowest degree subspace $M(0)$ is called the \emph{top level} of $M$. The smallest eigenvalue of $L(0)=\omega_{(1)}$ on $M(0)$ is called the \emph{top weight} of $M$, where $\omega$ is the conformal vector of $V$. 
\end{rem}

Since irreducible $V_{\sqrt{2}A_{n-1}}^+$-modules of $\sigma$-type
are also irreducible $V^{(n)}$-modules of $\sigma$-type, we obtain the following.

\begin{cor}\label{cor:4.3}
  There exist at least $(n+3)/2$ irreducible $V^{(n)}$-modules of $\sigma$-type if $n$ is odd,
  and there exist at least $n/2+3$ such modules if $n$ is even.
\end{cor}

\medskip

To show that $V^{(n)}$ has exactly $(n+3)/2$ (resp. $n/2+3$) irreducible modules of $\sigma$-type if $n$ is odd (resp. even), we consider some quotient algebra of the Zhu algebra as in \cite[Theorem 4.1]{JLY}.  

Let $ A(V^{(n)})=V^{(n)}/O(V^{(n)})$ be the Zhu algebra of $V^{(n)}$ (cf.~\cite{Z}).
For $u\in V^{(n)}$, we denote the image of $u$ in $A(V^{(n)})$ by $[u]$.
Recall that $A(V^{(n)})$ acts on the top level $M(0)$ of any $\N$-gradable module $M=\oplus_{i\geq 0} M(i)$ and there is a one to one correspondence between irreducible $A(V^{(n)})$-modules and irreducible $V^{(n)}$-modules \cite{Z}.

By Lemma 4.2 of \cite{JLY}, $A(V^{(n)})$ is generated by
$\{[ w^\pm(\al)] \mid \al\in \Delta_{A_{n-1}}^+\}$.
If $x\in V^{(n)}$ is an Ising vector then we have
\[
  [x]\l( [x]-\frac{1}{2}\r)\l( [x]-\frac{1}{16}\r)=0
\]
on any module over $A(V^{(n)})$ 
(cf.~\cite[Proposition 3.3]{DMZ}).

%for any $x\in \{ w^\pm(\al) \mid \al\in \Delta_{A_{n-1}}^+\}$.
Let $I^{(n)}$ be the two-sided ideal of $A(V^{(n)})$ generated by $[x]([x]-1/2)$ with
$x\in \{ w^\pm(\al) \mid \al\in \Delta_{A_{n-1}}^+\}$.
For an element $[a]$ in $A(V^{(n)})$, we denote its image in $A(V^{(n)})/I^{(n)}$ by
$\overline{[a]}$.
Then
\begin{equation}\label{eq:4.3}
  \overline{[w^\epsilon(\al)]}\l(\overline{[w^\epsilon(\al)]}-\frac{1}{2}\r)=0
\end{equation}
for $\al\in \Delta_{A_{n-1}}^+$.
%Since $V^{(n)}$ is generated by Ising vectors of $\sigma$-type,
%it follows from the fusion rules among $L(\shf,0)$-modules that
Clearly, a simple $\N$-gradable $V^{(n)}$-module $M=\oplus_{i\geq 0} M(i)$ is of $\sigma$-type
if and only if its top level $M(0)$ is a simple module over $A(V^{(n)})/I^{(n)}$.
% We have the following result.

For $\al\in \Delta_{A_{n-1}}^{+}$ and $\epsilon=\pm$, set
\begin{equation}\label{eq:4.4}
  s^\epsilon(\alpha):= \overline{[\mathbbm{1}-4 w^\epsilon(\al)]}.
%  \tilde{s}^{\al}=\overline{[\mathbbm{1}-4\tw^{\al}]}.
\end{equation}
By Lemma \ref{lem:prod}, Proposition 3.14 of \cite{JLY} and \eqref{eq:4.3}, one can
modify Lemma 3.3 of \cite{JL} to obtain the following relations.
\begin{equation}\label{eq:4.5}
\begin{array}{l}
  s^\epsilon(\alpha)s^\epsilon(\alpha)=1,~~~
  s^+(\alpha)s^-(\alpha)=s^-(\alpha)s^+(\alpha),
  \medskip\\
  s^\epsilon(\alpha)s^{\epsilon'}(\beta) = s^\epsilon(\beta)s^{\epsilon'}(\alpha)
  ~~\mbox{ if }~~ (\alpha\mymid \beta)=0,
  \medskip\\
  s^{\epsilon'}(\alpha) s^{\epsilon}(\beta) s^\epsilon(\alpha)
  = s^{-\epsilon\epsilon'}(r_\alpha(\beta))
  ~~\mbox{ if }~~ \abs{(\alpha\mymid \beta)}=1,
\end{array}
\end{equation}
where $r_\alpha$ denotes the reflection associated with $\alpha\in \Delta_{A_{n-1}}$.
Since $A(V^{(n)})$ is generated by $[w^\pm(\alpha)]$, 
$A(V^{(n)})/I^{(n)}$ is generated by $s^\pm(\alpha_i)$, $1\leq i<n$, by the relations in \eqref{eq:4.5},
where $\al_1, \dots, \al_{n-1}$ are the simple roots of $\Delta_{A_{n-1}}$ as before.
Indeed, one can verify that
\[
  s^+(\alpha_{i+1}) = s^-(\alpha_i)s^-(\alpha_{i+1})s^+(\alpha_i)s^-(\alpha_{i+1})s^-(\alpha_i)
  ~\mbox{ for }~ 1\leq i<n-1,
\]
so that $A(V^{(n)})/I^{(n)}$ is generated by $s^+(\alpha_1)$ and $s^-(\alpha_i)$, $1\leq i<n$.

We recall from \cite{GP} presentations of the Coxeter groups of type $B_n$ and $D_n$.
Let $W_n$ be the group generated by  $t$, $s_1,\cdots,s_{n-1}$ subject to the following relations
(cf.~Sections 1.4.1 of \cite{GP}):
\begin{equation}\label{eq:4.6}
\begin{array}{l}
  t^2=s_1^2=\cds =s_{n-1}^2=(ts_1)^4=1,~~ts_i=s_it ~\mbox{ for }~ i>1,
  \medskip\\
  s_is_{i+1}s_i=s_{i+1}s_is_{i+1} ~\mbox{ for }~ 1\leq i<n,~~
  s_is_j=s_js_i ~\mbox{ for }~ \abs{i-j}>1.
\end{array}
\end{equation}
Then $s_1,\cdots,s_{n-1}$ generate a subgroup $H_n$ of $W_n$ isomorphic to
the symmetric group $\mathfrak{S}_n$.
Set $t=t_0$ and $t_i=s_it_{i-1}s_i$ for $1\leq i<n$.
Then $t_0,\dots,t_{n-1}$ generate a normal subgroup $N_n$ isomorphic to an elementary
abelian 2-group of order $2^n$, and we have a semidirect product decomposition
$W_n=N_n\rtimes H_n\cong 2^n{:}\mathfrak{S}_n$.
Let $u=ts_1t$ and $W_n'$ the subgroup of $W_n$ generated by $u$, $s_1,\dots,s_{n-1}$.
Then one has
\begin{equation}\label{eq:4.7}
  us_2u=s_2us_2,~~~
  us_i=s_iu ~\mbox{ for }~ i\ne 2.
\end{equation}
Set $u_1=us_1$ and $u_i=s_iu_{i-1}s_i$ for $1< i<n$.
Then $u_i=t_0t_i$ for $1\leq i<n$ so that $u_1,\dots,u_{n-1}$ generate a subgroup $N_n'$
of index 2 in $N_n$ and we have a decomposition
$W_n'=N_n'\rtimes H_n\cong 2^{n-1}{:}\mathfrak{S}_n$ (cf.~Section 1.4.8 of \cite{GP}).
The groups $W_n$ and $W_n'$ are Coxeter groups of type $B_n$ and $D_n$, respectively.

\begin{lem}\label{lem:4.4}
  $A(V^{(n)})/I^{(n)}$ is isomorphic to a quotient of the group algebra of the
  Coxeter group $W_n'$ of type $D_n$.
\end{lem}

\pf
Define a map from the generators of the group algebra of $W_n'$ to
those of $A(V^{(n)})/I^{(n)}$ as follows.
\[
\begin{array}{l}
  u\longmapsto s^+(\alpha_1),~~~
  s_i\longmapsto s^-(\alpha_i) ~\mbox{ for }~ 1\leq i<n.
\end{array}
\]
By \eqref{eq:4.5}, one can directly verify that
this map preserves the defining relations \eqref{eq:4.6} and \eqref{eq:4.7} of
the Coxeter group $W_n'$ of type $D_n$ and can be uniquely extended to a homomorphism from
the group algebra to $A(V^{(n)})/I^{(n)}$. It  is surjective as it covers generators.
Thus the claim follows.
\qed
\medskip

\begin{comment}
One can verify that the homomorphism in Lemma \ref{lem:4.4} gives the correspondences
$u_1\longmapsto s^-(\alpha_1)s^+(\alpha_1)$ and
\[
  u_{i+1}\longmapsto s^-(\alpha_1)\cdots s^-(\alpha_i) (s^-(\alpha_{i+1})s^+(\alpha_{i+1}))
  s^-(\alpha_i)\cdots s^-(\alpha_1) ~\mbox{ for }~1\leq i<n-1.
\]
\end{comment}

\begin{df}
A pair $(\lam,\mu)=([\lam_1, \cdots, \lam_r], [\mu_1,\cdots,\mu_s])$ is called
a \emph{bipartition} of $n$ if $\lam$ and $\mu$ are partitions of non-negative integers such that
$|\lam|+|\mu|=\sum\limits_{i=1}^{r}\lam_i+\sum\limits_{j=1}^{s}\mu_j=n$.
\end{df}

It is known that irreducible modules over $W_n$ are in one to one correspondence
with bipartitions of $n$ (cf.~Section 3.4 of \cite{GP}).
For a bipartition $(\lam,\mu)$ of $n$, let $\chi_{(\lam,\mu)}\in {\rm Irr}(W_n)$ be
the irreducible character of $W_n$ constructed as in Definition 5.5.3 of \cite{GP}.
Denote by $\chi'_{(\lam,\mu)}$ the restriction of $\chi_{(\lam,\mu)}$ to $W'_n$.
Since $W'_n$ is a subgroup of $W_n$ of index 2, all irreducible characters of $W_n'$ are
obtained in the decompositions of $\chi'_{(\lam,\mu)}$, and we have the following classification
(cf.~Section 5.6.1 of \cite{GP}).

\begin{thm}\label{gp1}
  Let $(\lambda,\mu)$ be a bipartition of $n$.
  \\
  \textup{(1)}
  If $\lam\neq \mu$, then $\chi'_{(\lam,\mu)}=\chi'_{(\mu,\lam)}\in \mathrm{Irr}(W'_n)$.
  \\
  \textup{(2)}
  If $\lam=\mu$, then $\chi'_{(\lam,\lam)}=\chi'_{(\lam,+)}+\chi'_{(\lam,-)}$
  with two distinct characters $\chi'_{(\lam,\pm)}\in \mathrm{Irr}(W_n')$.
\end{thm}

By Lemma \ref{lem:4.4} and Theorem \ref{gp1}, each irreducible module $M$ over
$A(V^{(n)})/I^{(n)}$ corresponds to some character $\chi'_{(\lam,\mu)}$ or $\chi'_{(\lam,\pm)}$ if $\lam=\mu$.
We will not distinguish $\chi_{(\lam,\mu)}$ and $\chi'_{(\lam,\mu)}$ if $\lam\neq \mu$.
%So does $A(V_{\sqrt{2}A_{n-1}}^+)/I^{(n)}$.
%\blue{We denote the corresponding irreducible modules over $A(V^{(n)})/I^{(n)}$ by
%$(M, (\lam,\mu))$ in the case $\lambda \ne \mu$ and
%$(M, (\lam,\pm))$ if $\lam=\mu$. }
We may also assume that $\abs{\lam}\geq \abs{\mu}$.

%%%%%%%%%%%%%%%%%%%%%%%%%%%%%%%%%%%%%%%%%%%%%%%%%%%%%%%%%%%

\begin{lem}\label{lem:4.6}
  Let $(\lam,\mu)=([\lam_1, \cdots, \lam_r], [\mu_1,\cdots,\mu_s])$ be a bipartition of $n$
  which corresponds to an irreducible module over $A(V^{(n)})/I^{(n)}$.
  Then $r+s\leq 2$.
\end{lem}

\pf
Let $M$ be the irreducible $A(V^{(n)})/I^{(n)}$-module in the statement, and let $|\lam|=a$ and $|\mu|=b$ with $a\geq b$.
By Lemma 6.1.4 of \cite{GP}, we have
\begin{equation}\label{eq:4.8}
  \mathrm{Res}_{H_n}^{W_n} \chi_{(\lam,\mu)}
  =\mathrm{Ind}_{\mathfrak{S}_a\times \mathfrak{S}_b}^{\mathfrak{S}_n}(\chi_{\lam}\boxtimes \chi_{\mu})
  =\sum_{\nu\in\mathrm{Irr}\,\mathfrak{S}_n}c_{\lam\mu}^{\nu}\,\chi_{\nu},
\end{equation}
where $\chi_{\lam}\in \mathrm{Irr}(\mathfrak{S}_a)$ and
$\chi_{\mu}\in \mathrm{Irr}(\mathfrak{S}_b)$ are corresponding irreducible characters of
symmetric groups and $c_{\lam\mu}^{\nu}$ are the Littlewood-Richardson coefficients
(cf.~Section 6.1.6 of \cite{GP}).
In particular, $c_{\lam\mu}^{\nu}\neq 0$ if $\nu$ is a partition reordered from
$[\lam_1, \cdots, \lam_r, \mu_1,\cdots,\mu_s]$.
It follows from Lemma \ref{lem:4.4} that the subalgebra $S$ of $A(V^{(n)})/I^{(n)}$ generated by
$s^-(\alpha)$, $\alpha\in \Delta^+_{A_{n-1}}$, is isomorphic to a quotient of the group algebra of 
$H_n\cong \mathfrak{S}_n$.
By Theorem 4.1 of \cite{JLY}, the subalgebra $U$ of $V^{(n)}$ generated by
$w^-(\al)$, $\al\in \Delta^{+}_{A_{n-1}}$, is isomorphic to $K(A_{n-1},2)$,
and there is a homomorphism from the Zhu algebra $A(U)$ to the subalgebra $S$.
Therefore, every irreducible $H_n$-submodule of $M$ is labelled by a partition of $n$ of length
at most two (cf.~Lemmas 4.2, 4.8 and Theorem 4.16 of \cite{JL}).
If $\lambda\ne\mu$, then $M$ affords $\chi'_{(\lambda,\mu)}$ by
Theorem \ref{gp1}, and \eqref{eq:4.8} gives $r+s\leq2$.

Suppose $\lambda=\mu$.
Set $\widetilde H_n=tH_nt^{-1}=\langle u,s_2,\ldots,s_{n-1}\rangle \subset W'_n$.
The Ising vectors $w^+(\alpha_1)$, $w^-(\alpha_2),\ldots,w^-(\alpha_{n-1})$ generate another subVOA 
$\widetilde U\cong K(A_{n-1},2)$ by Theorem 4.1 of \cite{JLY}.
Applying the same argument to $\widetilde{U}$ shows that every irreducible
$\widetilde H_n$-submodule of $M$ is labelled by a partition of $n$ with length at most two.

Let $R$ be the irreducible $W_n$-module with character $\chi_{(\lambda,\lambda)}$.
By Theorem \ref{gp1}, $R|_{W'_n}=M\oplus M'$ for an irreducible $W'_n$-module $M'$.
Since $W'_n$ is normal in $W_n$, $t$ permutes these two summands,
and it must exchange them; otherwise one would be invariant under $W_n=\langle W'_n,t\rangle$.
The linear isomorphism $t:M'\longrightarrow M$ satisfies $t(hv)=(tht^{-1})(tv)$ 
for $h\in H_n$ and $v\in M'$.
Since every irreducible $\widetilde H_n$-submodule of $M$ is labelled by a partition of $n$ with length 
at most two, the same holds for every irreducible $H_n$-submodule of $M'$.
Thus all irreducible $H_n$-submodules of $R$ are labelled by partitions of $n$ with length at most two.
Let $\nu$ be the partition obtained by reordering 
$[\lambda_1,\dots,\lambda_r,\mu_1,\dots,\mu_s]=[\lambda_1,\dots,\lambda_r,\lambda_1,\dots,\lambda_r]$.
Since Littlewood-Richardson coefficients satisfy $c_{\lambda,\lambda}^\nu\ne 0$, we have $r+s=2r\leq2$.
\qed
\medskip

Before we give some bipartitions of $n=4$,
which correspond to the irreducible modules of $A(V^{(4)})/I^{(4)}$,
we will first prove the following result.

\begin{lem}\label{lem:4.7}
  $V^{(4)}$ is simple and isomorphic to $K(D_4,2)\cong V_{\sqrt{2}A_{3}}^+$.
\end{lem}

\pf
Suppose $V^{(4)}$ is not simple.
Then there exists a non-trivial ideal $J$ of $V^{(4)}$ with the top weight $h\in \Z_{\geq 3}$
such that $V^{(4)}/J\cong V_{\sqrt{2}A_3}^+$.
Recall that the set $E_{V^{(4)}}$ of Ising vectors of $V^{(4)}$ can be identified with the set
$\{ w^-(\alpha) \mid \alpha \in \Delta_{D_4}^+\}$ and
$G=\la \sigma_e\mid e\in E_{V^{(n)}}\ra \cong F_4{:}\mathfrak{S}_4$.
Let $\{\beta_1,\beta_2,\beta_3,\beta_4\}$ be a simple roots of $\Delta_{D_4}$ such that
$(\beta_2\mymid \beta_i)=-1$ for $i\ne 2$.
Let $\Delta_1$ be the root system generated by $\{\beta_1,\beta_2,\beta_3\}$.
Then $\Delta_1$ is a root system of type $A_3$.
Let $\mathcal{E}'=\{ w^-(\alpha) \mid \alpha\in \Delta_1^+\}$ and $W=\la \mathcal{E}'\ra$.
It follows from \cite[Theorem 4.1]{JLY} and Remark \ref{rem:3.9} that $W$ is isomorphic to
$K(A_3,2)\cong V_{\sqrt{2}A_2}^+$.
The corresponding subgroup of $G$ is $H=\la \sigma_e \mid e\in \mathcal{E}'\ra\cong \mathfrak{S}_4$.
By Lemma \ref{lem:4.1} and Remark \ref{rem:4.2}, $W\cong V_{\sqrt{2}A_2}^+$ has three
irreducible modules of $\sigma$-type with top weights $0$, $2/3$ and $1$.
Applying Lemma \ref{lem:3.14} with $\abs{I_G}=12$, $\abs{I_H}=6$,
$k_G=8$, $k_H=4$ and $\lambda=1$, we obtain $h\leq 3/2$, a contradiction.
Therefore, $V^{(4)}$ is simple and isomorphic to $K(D_4,2)\cong V_{\sqrt{2}A_3}^+$
by Theorem \ref{thm:3.7}.
\qed
\medskip

By Lemmas \ref{lem:4.1} and \ref{lem:4.7}, $A(V^{(4)})/I^{(4)}$ has five irreducible modules
which are the top levels of the following five irreducible modules over $V^{(4)}$ of $\sigma$-type.
\[
  V_{\sqrt{2}A_3}^+, ~~~
  V_{\sqrt{2}A_3}^-, ~~~
  V_{\sqrt{2}A_3+2\Lambda_3},~~~
  V_{\sqrt{2}A_3+4\Lambda_3}^+,~~~
  V_{\sqrt{2}A_3+4\Lambda_3}^-.
\]
%which we denote by $M(i)$, $i=1,2,3,4,5$, respectively.
%Then we have the following lemma.

\begin{lem}\label{lem:4.8}
  The correspondence between the irreducible modules of
  $A(V^{(4)})/I^{(4)}$ and bipartitions of $4$ is as follows.
  \[
  \begin{array}{l}
     V_{\sqrt{2}A_3}^+ \longleftrightarrow ([4], \varnothing),~~~~~
     V_{\sqrt{2}A_3}^- \longleftrightarrow ([3,1], \varnothing),
     \medskip\\
     V_{\sqrt{2}A_3+2\Lambda_3} \longleftrightarrow  ([3], [1]),~~~~~
     V_{\sqrt{2}A_3+4\Lambda_3}^\pm \longleftrightarrow  ([2],\pm)
     \end{array}
  \]
\end{lem}

\pf
Consider the subalgebra $S$ of $A(V^{(4)})/I^{(4)}$ generated by
$s^-(\alpha)$, $\alpha\in \Delta_{A_3}^+$.
Then $S$ is a quotient of the group algebra of $H_4\cong \mathfrak{S}_4$.
By explicit calculations, one can determine the characters of the top levels of irreducible
$V^{(4)}\cong V_{\sqrt{2}A_3}^+$-modules of $\sigma$-type as $H_4$-module as follows\footnote{%
Strictly speaking, there is an ambiguity to assign labelings $\pm$ to $V_{\sqrt{2}A_3+4\Lambda_3}^\pm$.
We just assign $\pm$ so that we can get the decomposition as the displayed.}.
\[
\begin{array}{ccccccccc}
  V_{\sqrt{2}A_3}^+ && V_{\sqrt{2}A_3}^- && V_{\sqrt{2}A_3+2\Lambda_3} && V_{\sqrt{2}A_3+4\Lambda_3}^+ &&
  V_{\sqrt{2}A_3+4\Lambda_3}^-
  \medskip\\
  \chi_{[4]} && \chi_{[3,1]} && \chi_{[4]}+\chi_{[3,1]} && \chi_{[4]}+\chi_{[2,2]} && \chi_{[3,1]}
\end{array}
\]
Now the assertion follows from Lemma \ref{lem:4.6} and the rule \eqref{eq:4.8}.
\qed

\begin{lem}\label{lem:4.9}
%  $A(V^{(n)})/I^{(n)}\cong A(V_{\sqrt{2}A_{n-1}}^+)/I^{(n)}$.
  The number of irreducible $V^{(n)}$-module of $\sigma$-type is
  exactly $(n+3)/2$ if $n$ is odd and $n/2+3$ if $n$ is even.
  In particular, the simple quotient of $V^{(n)}$ acts on each irreducible
  $V^{(n)}$-module of $\sigma$-type.
\end{lem}

\pf
Since the top level of a $V^{(n)}$-module of $\sigma$-type is a module over $A(V^{(n)})/I^{(n)}$,
it is enough to count the number of inequivalent irreducible $A(V^{(n)})/I^{(n)}$-modules.
We will prove that the possible bipartitions of $n$ which correspond to
the irreducible modules over $A(V^{(n)})/I^{(n)}$ are
\[
  ([n],\varnothing),~~~
  ([n-1,1],\varnothing),~~~
  ([n-i],[i]),~~ 1\leq i\leq \dfr{n-1}{2}
\]
if $n\in 2\Z+1$, and
\[
  ([n],\varnothing),~~~
  ([n-1,1],\varnothing),~~~
  ([n/2],\pm), ~~~
  ([n-i],[i]),~~ 1\leq i\leq \dfr{n}{2}-1
\]
if $n\in 2\Z$.
Then the claim follows from Corollary \ref{cor:4.3} and Lemma \ref{lem:4.4}.

We proceed by induction on $n\geq 4$.
The statement is true for $n=4$ by Lemmas \ref{lem:4.7} and \ref{lem:4.8}.
Suppose that the statement holds for $V^{(m)}$ with $4\leq m\leq n-1$.
Let $(\lam,\mu)=([\lam_1, \cdots, \lam_r], [\mu_1,\cdots,\mu_s])$ with $\abs{\lam}>\abs{\mu}$
be a bipartition of $n$ which corresponds to an irreducible module over $A(V^{(n)})/I^{(n)}$.
%We may assume that $\abs{\lam}\neq \abs{\mu}$.
By Lemma \ref{lem:4.6}, $r+s\leq 2$.
The bipartitions satisfying $r+s\leq 2$ are
\[
  ([n],\varnothing),~~~
%  ([n-1,1],\varnothing),~~~
  ([n-i,i], \varnothing),~~~
  ([n-i],[i]),~~  1\leq i\leq \dfr{n-1}{2}
\]
if $n$ is odd, and
\[
  ([n],\varnothing),~~
  ([n/2,n/2], \varnothing),~~
  ([n/2],\pm),~~
%  ([n-1,1],\varnothing),~~
  ([n-i,i],\varnothing),~~
  ([n-i],[i]),~~1\leq i\leq \dfr{n}{2}-1
\]
if $n$ is even.
By Corollary \ref{cor:4.3}, it suffices to show that bipartitions
$([n-i,i], \varnothing)$ for $2\leq i\leq n/2$ do not correspond to any irreducible
modules over $A(V^{(n)})/I^{(n)}$.
% if $n$ is odd, and so do those $([n-i,i],\varnothing)$, $2\leq i< n/2$ if $n$ is even.

Suppose false and $([n-i,i], \varnothing)$ corresponds to an irreducible module over
$A(V^{(n)})/I^{(n)}$ for some $2\leq i\leq n/2$.
Recall the following branching rule (cf.~Section 6.1.9 of \cite{GP}).
\begin{equation}\label{eq:4.9}
  \mathrm{Res}_{W_{n-1}}^{W_{n}} \chi_{(\lam,\mu)}
  = \sum_{d\in I(\lam)} \chi_{(\lam^{(d)},\mu)} + \sum_{d\in I(\mu)}\chi_{(\lam, \mu^{(d)})},
\end{equation}
where $I(\lam)$ is the set of all $d\in \{1,2,\cdots,r\}$ such that $d=r$,
or $d< r$ and $\lam_d>\lam_{d+1}$, and $\lam^{(d)}$ with $d\in I(\lam)$ denotes
the partition of $|\lam|-1$ which is obtained from $\lam$ by decreasing $\lam_d$ by $1$.
If $i>2$, then $\chi_{([n-i, i-1], \varnothing)}$ will appear in the right hand side of
\eqref{eq:4.9} which contradicts the inductive assumption.
If $i=2$, then $\chi_{([n-3, 2], \varnothing)}$ will appear in the right hand side of \eqref{eq:4.9})
which again contradicts the inductive assumption.
\qed

\medskip

We will prove that $V^{(n)}\cong V_{\sqrt{2}A_{n-1}}^+$ by induction on $n$. 
Let $\Delta_{A_{n-2}}$ be the root subsystem  generated by the simple roots 
$\{\alpha_1,\dots,\alpha_{n-2}\}$ and $V^{(n-1)}$ the subVOA of $V^{(n)}$ generated by
$\{ w^\pm(\alpha) \mid \alpha \in \Delta_{A_{n-2}}^+\}$. 
 Let $U=\com(V^{(n-1)},V^{(n)})$ be the commutant of $V^{(n-1)}$ in $V^{(n)}$.
\begin{comment}
Let $\w$ and $\eta$ be the conformal vectors of $V^{(n)}$ and $V^{(n-1)}$, respectively.
Since $V^{(n)}$ is of OZ-type, $\xi:= \w-\eta$ is the conformal vector of $U$  (cf.~\cite{FZ}) and $\la \xi \ra$ is a simple Virasoro VOA of central charge $1$ by \eqref{eq:4.10}.
\end{comment}
By induction hypothesis, we may assume that $V^{(n-1)}$ is simple and  $V^{(n-1)}\cong V_{\sqrt{2}A_{n-2}}^+$.

Recall the vector $\Lambda_{n-1}\in (\sqrt{2}A_{n-1})^*$ in \eqref{eq:4.1} and set $\gamma =2n\Lambda_{n-1}$. 
Note that $\Delta_{A_{n-2}}$ is orthogonal to $\gamma$ and $\Z\gamma= \sqrt{2}A_{n-1} \cap \Z\Lambda_{n-1}$; thus, 
$\sqrt{2}A_{n-1}$ contains a full rank sublattice $\sqrt{2}A_{n-2} \perp \Z\gamma$ and $\com( V_{\sqrt{2}A_{n-2}}^+, V_{\sqrt{2}A_{n-1}}^+)\cong V_{\Z\gamma}^+$. 

Since $V^{(n)}$ has the simple quotient isomorphic to $V_{\sqrt{2}A_{n-1}}^+$, the commutant $U$ has the simple quotient
isomorphic to $V_{\Z\gamma}^+$.

\subsection{Lattice VOA $V_{\Z\gamma}$ and its subVOA $V_{\Z\gamma}^+$}

Consider the lattice VOA $V_{\Z\gamma}$ and its subVOA $V_{\Z\gamma}^+$, 
where $\gamma =2n\Lambda_{n-1}$.
Note that $(\gamma\mymid \gamma)/2=n(n-1)$ is not a square of positive integers.  

We denote by $M(1)$ the Heisenberg subalgebra of $V_{\Z\gamma}$ generated by
$\gamma=\gamma_{(-1)}\vac$ and by $M(1)^+$ the fixed point subalgebra inside $V_{\Z\gamma}^+$.
Set $\gamma^*=(1/2n(n-1))\gamma$.
Then $(\Z\gamma)^*=\Z \gamma^*$ and the irreducible $V_{\Z\gamma}$-modules are given by
$V_{\Z\gamma+i\gamma^*}$, $0\leq i < 2n(n-1)$.
We denote by $M(1,\alpha)$ the $M(1)$-submodule of $V_{\Z\gamma^*}$ generated by
$e^{\alpha}\in \C[\Z \gamma^*]$ for $\alpha\in \Z \gamma^*$.

Let $V(c,h)$ be the Verma module over the Virasoro algebra with central charge $c$ and
highest weight $h$ and $L(c,h)$ its simple quotient.
Recall from \cite{KR} that if $h=n^2/4$ for some $n\in \Z_{\geq 0}$ then $V(1,h)$ is reducible and we have the following
short exact sequence.
\begin{equation}\label{eq:4.10}
  0\longto V\l( 1,(n+2)^2/4\r)\longto V\l( 1,n^2/4\r)\longto  L\l(1,n^2/4\r) \longto 0.
\end{equation}
Otherwise, $V(1,h)=L(1,h)$.
Let $0\neq \alpha\in \Z\gamma^*$. Then $(\alpha\mymid \alpha)/2$ is not of
the form $n^2/4$ for any integer $n$; thus, $M(1,\alpha)$ is simple and isomorphic to
$V(1,(\alpha\mymid\alpha)/2)=L(1,(\alpha\mymid\alpha)/2)$
as a module over the Virasoro algebra at $c=1$ (cf.~\cite{DG,KR}).
Therefore we have the following decompositions as modules over the Virasoro algebra
 for $1\leq i< 2n(n-1)$ (cf.~Proposition 2.2 of \cite{DG}),
\begin{equation}\label{eq:4.11}
  V_{\Z\gamma+i\gamma^*}
  = \bigoplus_{p\in \Z} M(1,p\gamma+i\gamma^*)
  = \bigoplus_{p\in \Z} L(1,h_{n,p,i}),~~~
  h_{n,p,i}=\dfr{(2n(n-1)p+i)^2}{4n(n-1)},
\end{equation}
whereas $M(1)=M(1,0)$ and $M(1)^\pm$ are irreducible.
We also have the following decompositions (cf.~Theorem 2.7 and Lemma 2.10 of \cite{DG}).
\begin{equation}\label{eq:4.12}
\begin{array}{l}
  \ds
  V_{\Z\gamma}
  = \bigoplus_{p\in \Z} M(1,p\gamma)
  = M(1)\oplus \bigoplus_{p\geq 1} L(1,p^2n(n-1))^{\oplus 2},~~~
  M(1)=\bigoplus_{p\geq 0} L(1,p^2),
  \medskip\\
  \ds
  V_{\Z\gamma}^+
%  = \bigoplus_{p\in \Z} M(1,p\gamma)
  = M(1)^+\oplus \bigoplus_{p\geq 1} L(1,p^2n(n-1)),~~~
  M(1)^+=\bigoplus_{p\geq 0} L(1,4p^2),
  \medskip\\
  \ds
  V_{\Z\gamma}^-
%  = \bigoplus_{p\in \Z} M(1,p\gamma)
  = M(1)^-\oplus \bigoplus_{p\geq 1} L(1,p^2n(n-1)),~~~
  M(1)^-=\bigoplus_{p\geq 0} L(1,(2p+1)^2).
\end{array}
\end{equation}
It is shown in Theorem 2.9 of \cite{DG} that $V_{\Z \gamma}^+$ is generated by its conformal
vector together with the following primary vectors of weight 4 and $n(n-1)$.
\begin{equation}\label{eq:4.13}
  u = 2{\gamma_{(-1)}^*}^2\gamma_{(-1)}^2\vac-2\gamma^*_{(-3)}\gamma
      + 3\gamma_{(-2)}^*\gamma_{(-2)}\vac,~~~~
  v = e^{\gamma}+e^{-\gamma}.
\end{equation}
It is also shown in (loc.~cit.) that the conformal vector and $u$ generates the subVOA
$M(1)^+$ in $V_{\Z\gamma}^+$.

\subsection{$V^{(n)}$ as $V^{(n-1)}\otimes U$-modules}

Denote the conformal vectors of $V^{(n-1)}$ and $U$ by $\eta$ and $\xi$, respectively.
Let 
\begin{equation}\label{eq:4.14}
  \pi: U\longto V_{\Z\gamma}^+
\end{equation}
be the natural surjection and take preimages $\tilde{u}$ and $\tilde{v}$ of $u$ and $v$ defined in
\eqref{eq:4.13} with respect to the map $\pi$.
Since $\tilde{u}$ is of weight four and the weight two subspaces of $V^{(n)}$ and its simple
quotient $V_{\sqrt{2}A_{n-1}^+}$ are isomorphic, we have $\xi_{(3)}\tilde{u}=0$.
Therefore, by replacing $\tilde{u}$ by $\tilde{u}-(1/6)\xi_{(0)}\xi_{(2)}\tilde{u}$ if necessary,
we may assume that $\tilde{u}$ is a primary vector for $\la \xi\ra$ of weight 4.

Now suppose that $V^{(n)}$ is not simple. 
Since $V^{(n)}$ itself is of $\sigma$-type, by Lemma \ref{lem:4.9}, there are only finitely many irreducible modules of $\sigma$-type and $V^{(n)}$ has a composition series
\begin{equation}\label{eq:4.15}
  V^{(n)}=J_0 \supset J_1\supset \cds \supset J_k=0
\end{equation}
with $k>1$ and  the composition factors $J_i/J_{i+1}$, $0\leq i<k$, are isomorphic to
one of the following irreducible $V^{(n)}$-modules of $\sigma$-type.
\begin{equation}\label{eq:4.16}
  V_{\sqrt{2}A_{n-1}}^{\pm},~~
  V_{\sqrt{2}A_{n-1}+2i\Lambda_{n-1}},~~1\leq i< \frac{n}{2},~
  \mbox{ and} ~~ V_{\sqrt{2}A_{n-1}+n\Lambda_{n-1}}^{\pm} ~\mbox{ if $n$ is even}.
\end{equation}
In practice, the top factor $J_{0}/J_{1}=V^{(n)}/J_{1}$ is isomorphic to $V_{\sqrt{2}A_{n-1}}^+$
but the remaining factors $J_{i}/J_{i+1}$ for $i\geq 1$ are not isomorphic to
$V_{\sqrt{2}A_{n-1}}^\pm$ since $V^{(n)}$ is of OZ-type.
Seen as a module over the conformal subalgebra $V^{(n-1)}\tensor U$ of $V^{(n)}$,
each factor is also a module over the simple quotient of $U$ isomorphic to $V_{\Z\gamma}^+$
again by Lemma \ref{lem:4.9}.
Therefore, it is a direct sum of irreducible $V_{\Z\gamma}^+$-modules
among the followings.
\begin{equation}\label{eq:4.17}
  V_{\Z\gamma}^{\pm}, ~~~
  V_{\Z\gamma+\gamma/2}^{\pm},~~~
  V_{\Z\gamma+i\gamma^*},~~ 1\leq i<n(n-1).
\end{equation}
It follows from \eqref{eq:4.11} and \eqref{eq:4.12} that each factor $J_{i}/J_{i+1}$
in \eqref{eq:4.15} is semisimple as a $\la \xi\ra$-module.

\begin{lem}\label{lem:4.10}
  There is no highest weight vector for $\la \eta\ra\tensor \la \xi\ra$ of weights $(0,m^2)$
  with $m\in \Z_{\geq 0}$ in the composition factors $J_{i}/J_{i+1}$ in \eqref{eq:4.15}
  for $1\leq i<k$.
\end{lem}

\pf
Suppose we have a highest weight vector $x$ for $\la \eta\ra\tensor \la \xi\ra$
with highest weights $(0,m^2)$ in a factor $J_{i}/J_{i+1}$ for $i\geq 1$.
Then $x$ generates a $\la \xi\ra$-submodule isomorphic to $L(1,m^2)$ in $J_{i}/J_{i+1}$.
Consider the $V_{\Z\gamma}^+$-submodule $M$ generated by $x$ in $J_{i}/J_{i+1}$.
Since $h_{n,p,i}$ in \eqref{eq:4.11} is not a square of integers for $1\leq i<2n(n-1)$,
$M$ is a direct sum of irreducible $V_{\Z\gamma}^+$-submodules isomorphic to $V_{\Z\gamma}^\pm$.
In fact, if $m$ is even, then $M$ is a direct sum of $V_{\Z\gamma}^+$-submodules isomorphic
to $V_{\Z\gamma}^+$, and if $m$ is odd, then $M$ is a direct sum of submodules isomorphic to
$V_{\Z\gamma}^-$ by \eqref{eq:4.12}.
Since $\eta$ and $\xi$ are orthogonal, the $\la \xi\ra$-weights and the $\la \w\ra$-weights
are the same on $M$.
Therefore, if $M$ contains a submodule isomorphic to $V_{\Z\gamma}^+$,
then $\dim (V^{(n)}/J_{i+1})_0\geq 2$, and if $M$ contains a submodule isomorphic to
$V_{\Z\gamma}^-$, then $\dim (V^{(n)}/J_{i+1})_1>0$.
Either way, we obtain a contradiction.
\qed

\begin{lem}\label{lem:4.11}
  The $\langle \xi \rangle$-submodule $X$ of $U$ generated by $\tilde{u}$ is simple.
\end{lem}

\pf
Suppose that $X$ is not simple.
Then by \eqref{eq:4.10} $X$ has a singular vector $x$ of weight $9$.
By Lemma \ref{lem:4.9}, there are finitely many irreducible $V^{(n)}$-modules of $\sigma$-type.
Since $u=\pi(\tilde{u})$ generates a simple module over $\la \xi\ra$ in the simple quotient
$V^{(n)}/J_{1}\cong V_{\sqrt{2}A_{n-1}}^+$, the singular vector $x$ is in the maximal ideal $J_{1}$
and there exists $0< i<k$ such that $x\in J_{i}$ and $x \not\in J_{i+1}$.
Then the image of $x$ in $J_{i}/J_{i+1}$ determines a highest weight vector for
$\la \eta\ra\tensor \la \xi\ra$ with the highest weight $(0,3^2)$, contradicting Lemma \ref{lem:4.10}.
Thus $X$ is simple.
\qed

\begin{comment}
It follows from Lemma \ref{lem:4.9} that the simple quotient $V_{\sqrt{2}A_{n-1}}^+$ of $V^{(n)}$
and its subalgebra $V_{\Z\gamma}^+$ act on the composition factors $J_{i}/J_{i+1}$.
Since the weight of $\ol{x}$ with respect to $\xi$ is 9, it follows from the decompositions
\eqref{eq:4.11} and \eqref{eq:4.12} that the $V_{\Z\gamma}^+$-submodule generated by $\ol{x}$
is isomorphic to $V_{\Z\gamma}^-$ with the conformal weight 1.
Then $U$ contains an element of conformal weight 1 with respect to the conformal vector
$\omega$ of $V^{(n)}$ as $\tilde{v} \in U$ is in the commutant of $V^{(n-1)}$ in $V^{(n)}$,
but this contradicts our assumption that $V^{(n)}$ is of OZ-type.
Thus $X$ is simple as a $\la \xi\ra$-module.
\end{comment}

\begin{lem}\label{lem:4.12}
  The subalgebra $Y$ of $U$ generated by $\omega$ and $\tilde{u}$ is simple and
  isomorphic to $M(1)^+$.
\end{lem}

\pf
Let $X\cong L(1,4)$ be the $\la \xi\ra$-submodule generated by $\tilde{u}$ as in Lemma
\ref{lem:4.11}.
Set $P^0=\la \xi\ra$, $P^1=\la \xi\ra\oplus X\cong L(1,0)\oplus L(1,4)$ and
$P^{i+1}:=P^1\cd P^i$ for $i\geq 1$.
We will show that
\begin{equation}\label{eq:4.18}
  P^i\cong \bigoplus_{0\leq j\leq i} L(1,(2j)^2)\cong P^{i-1}\oplus L(1,(2i)^2)
\end{equation}
as a $\la \xi\ra$-module for $i\geq 1$.
Clearly, this is true for $i=1$.
Suppose $P^i$ has the decomposition as in \eqref{eq:4.18} for $i\geq 1$.
Consider the image $\pi(P^{i+1})=\pi(P^i)+\pi(X)\cd \pi(P^i)$ in the simple quotient
$V_{\Z\gamma}^+$ where $\pi$ is the surjection in \eqref{eq:4.14}.
Since $V_{\Z\gamma}^+$ is a semisimple $\la \xi\ra$-module and has a decomposition as in
\eqref{eq:4.12}, it follows from the fusion rules
\begin{equation}\label{eq:4.19}
\begin{split}
  L(1,2^2)\times L(1,j^2)
  =& ~L(1,(j-2)^2) + L(1,(j-1)^2) +L(1,j^2)
  \medskip\\
  &~~~ + L(1,(j+1)^2)+L(1,(j+2)^2)
\end{split}
\end{equation}
of $L(1,0)$-modules for $j\geq 2$
(cf.~Theorem 4.7 of \cite{DJ} and Theorem 3.4 of \cite{Ml}) that
\begin{equation}\label{eq:4.20}
  \pi(P^{i+1})
  \cong \bigoplus_{0\leq j\leq i+1} L(1,(2j)^2)
  \cong \pi(P^i)\oplus L(1,(2i+2)^2)
\end{equation}
as $\la \xi\ra$-modules.
Let $u^{i+1}$ be a highest weight vector of $L(1,(2i+2)^2)$ in \eqref{eq:4.20} and
$\tilde{u}^{i+1}$ its preimage in $P^{i+1}$.
%We will show that the $\la \xi\ra$-submodule generated by $\tilde{u}^{i+1}$ is simple.
Let $M=P^{i+1}\cap \ker\, \pi$.
Since $\ker\,\pi \subset J_{1}$,  we have $M\subset J_{1}$, where $J_{1}$ is the
maximal ideal of $V^{(n)}$ in \eqref{eq:4.15}.
By definition, we have $P^i\subset P^{i+1}$, and by inductive assumption, we have
$P^i\cap M=0$.
Let $Q$ be the irreducible $\la \xi\ra$-submodule of $P^i$ isomorphic to $L(1,(2i)^2)$,
i.e., $P^i=P^{i-1}\oplus Q$.
It follows from the fusion rule \eqref{eq:4.19} and the inductive assumption that
$\tilde{u}^{i+1}\in X\cd Q$ and $M\subset X\cd Q$.
Suppose $M\ne 0$ and let $h\in \Z_{>1}$ be the top weight of $M$, i.e.,
$M=\bigoplus_{n\geq 0} M_{n+h}$ with $M_h\ne 0$.
If $h\geq (2i+2)^2$, then $\tilde{u}^{i+1}$ is a highest weight vector for $\la \xi\ra$
and it follows from the structure \eqref{eq:4.10} of Verma modules and Lemma \ref{lem:4.10}
that $\tilde{u}^{i+1}$ generates a simple $\la \xi\ra$-submodule isomorphic to $L(1,(2i+2)^2)$
in $P^{i+1}$ and hence $M$ splits in $P^{i+1}$.
\[
  P^{i+1} = P^i \oplus L(1,(2i+2)^2)\oplus M.
\]
Let $J_{t}$ be an ideal in \eqref{eq:4.15} such that $M_h\not\subset J_{t}$ and
$M/J_{t}$ is semisimple as a $\la \xi\ra$-module.
Then $L(1,h) \subset M/J_{t}$.
Since $M\subset X\cd Q$, we obtain a non-zero $L(1,0)$-intertwining operator of type
\[
  L(1,2^2)\times L(1,(2i)^2)\to L(1,h)
\]
in $M/J_{t}\subset (X\cd Q)/J_{t}$ with $h\geq (2i+2)^2$.
Then by the fusion rule \eqref{eq:4.19} we must have $h=(2i+2)^2$ but the fusion coefficient
of this type is one and we already have $\tilde{u}^{i+1}$ in $X\cd Q$, a contradiction.
Therefore, $h<(2i+2)^2$.
Each non-zero vector of $M_h$ is a highest weight vector for $\la \xi\ra$.
By Lemma \ref{lem:4.10}, $h$ is not a perfect square, and the $\la \xi\ra$-submodule
generated by $M_h$ is a direct sum of irreducible submodules isomorphic to $L(1,h)$ by \cite{KR}.
%Let $Q$ be the irreducible $\la \xi\ra$-submodule of $P^i$ isomorphic to $L(1,(2i)^2)$.
Since $M\subset X\cd Q$, we obtain a non-zero $L(1,0)$-intertwining operator of type
\[
  L(1,2^2)\times L(1,(2i)^2)\to L(1,h)
\]
for a non-square $h\in \Z_{>1}$, which contradicts the fusion rule \eqref{eq:4.19}.
Therefore, $M=0$
%and $\tilde{u}^{i+1}$ generates a $\la \xi\ra$-submodule isomorphic to $L(1,(2i+2)^2)$,
and hence $P^{i+1}$ has the desired decomposition.
Since $Y=\cup_{i\geq 0}P^i$, we conclude that $Y\cong M(1)^+$ by \eqref{eq:4.12}.
\qed

\begin{lem}\label{lem:4.13}
  The subVOA $U=\com(V^{(n-1)},V^{(n)})$ contains a
  conformal subalgebra isomorphic to $V_{\Z\gamma}^+$.
  In particular, $V^{(n)}$ is a conformal extension of $V_{\sqrt{2}A_{n-2}}^+\tensor V_{\Z\gamma}^+$.
\end{lem}

\pf
Let $Y$ be the subalgebra of $U$ isomorphic to $M(1)^+$ as in Lemma \ref{lem:4.12}.
Recall the surjection $\pi$ in \eqref{eq:4.14}.
Then $\pi(Y)=M(1)^+$.
The simple quotient of $U$ is isomorphic to $V_{\Z \gamma}^+$ which is semisimple as
a $\la \xi\ra$-module as in \eqref{eq:4.12}.
Denote the complement of $M(1)^+$ in $V_{\Z\gamma}^+$ by $T$ so that we have
a decomposition $V_{\Z\gamma}^+=M(1)^+\oplus T$ as a $\la \xi\ra$-module.
Set $\tilde{T}=\pi^{-1}(T)$.
Then $U=Y\oplus \tilde{T}$ as a $\la \xi\ra$-module.
First we will show that $\tilde{T}$ is a semisimple $\la \xi\ra$-module.
The following argument is a slight modification of the proof of Proposition 5.11 of \cite{DLTYY}.
There is a filtration
\begin{equation}\label{eq:4.21}
  \tilde{T}=\tilde{T}_0\supset \tilde{T}_1\supset \cds \supset \tilde{T}_k=0
\end{equation}
such that $\tilde{T}_1=\ker\,\pi$ is the maximal ideal,
$\tilde{T}_i$ are ideals of $U$ and factors $\tilde{T}_i/\tilde{T}_{i+1}$ are semisimple as
$\la \xi\ra$-modules by \eqref{eq:4.15}.
Moreover, each irreducible $\la \xi\ra$-submodule of $\tilde{T}_i/\tilde{T}_{i+1}$ is
a highest weight module whose highest weight is a non-square integer by \eqref{eq:4.11},
\eqref{eq:4.12} and Lemma \ref{lem:4.10}.
In particular, any highest weight vector in $\tilde{T}_i/\tilde{T}_{i+1}$ generates a
simple $\la \xi\ra$-submodule by the structures of Verma modules at $c=1$ (cf.~\cite{KR}).
Let $h$ be the top weight of $\tilde{T}$ and $(\tilde{T})_h$ the top level of $\tilde{T}$.
Let $A$ be the $\la \xi\ra$-submodule generated by $(\tilde{T})_h$.
Then $A$ is a direct sum of irreducible highest weight modules isomorphic to $L(1,h)$.
Consider the following exact sequence of $\la \xi\ra$-modules.
\begin{equation}\label{eq:4.22}
  0\longto A \longto \tilde{T} \longto \tilde{T}/A\longto 0.
\end{equation}
Since $\tilde{T}$ has a $\N$-grading with finite dimensional homogeneous components,
we can take its restricted dual.
So taking the dual modules, i.e., restricted duals\footnote{%
Here $M^*$ stands for the restricted dual of $\N$-gradable module $M$.},
we obtain the following exact sequence of $\la \xi\ra$-modules.
\begin{equation}\label{eq:4.23}
  0 \longto (\tilde{T}/A)^* \longto {\tilde{T}\, }^* \longto A^* \longto 0.
\end{equation}
Since the top weight $h$ of ${\tilde{T}\,}^*$ is not a perfect square,
its top level generates a semisimple $\la \xi\ra$-submodule isomorphic to $A^*\cong A$
as $L(1,h)^*\cong L(1,h)$.
Therefore, the extension \eqref{eq:4.23} splits and so does the extension \eqref{eq:4.22}.
We can apply the same argument to the complement of $A$ in $\tilde{T}$ which is isomorphic to
$\tilde{T}/A$, and since the top weight of $\tilde{T}/A$ is strictly larger than $h$,
we see by induction that $\tilde{T}$ is a direct sum of highest weight $\la \xi\ra$-submodules.

Thus, we can decompose $\tilde{T}=T'\oplus \ker\,\pi$ such that $\pi(T')=T$ as
$\la \xi\ra$-modules.
By \eqref{eq:4.12}, we have the following decomposition as a $\la \xi\ra$-module.
\[
  T'=\bigoplus_{p\geq 1} T'[p],~~~T'[p]\cong L(1,p^2n(n-1)).
\]
We prove that each $T'[p]$ is an irreducible $Y$-module.
Recall some  fusion rules of $L(1,0)$-modules from \cite{DJ}.
Let $m$, $p$, $q$ be positive integers such that $p$ is not a square.
Then it is shown in Theorem 4.7 of \cite{DJ} that
\begin{equation}\label{eq:4.24}
  \dim \binom{L(1,q)}{L(1,m^2)~~~L(1,p)}_{L(1,0)}=\delta_{p,q}.
\end{equation}
Since $V^{(n)}$ is of OZ-type, there is no irreducible $V_{\Z\gamma}^+$-submodule isomorphic to
$V_{\Z\gamma}^\pm$ in the composition factors of $\ker\,\pi$.
Then it follows from the decompositions in \eqref{eq:4.11} and \eqref{eq:4.12} that
there is no overlapping between irreducible $\la \xi\ra$-submodules  of $T'$ and $\ker\,\pi$;
thus, the multiplicity of $L(1,p^2n(n-1))$ in $\tilde{T}$ is just one.
Therefore, the fusion rules \eqref{eq:4.24} guarantee $Y\cd T'[p]\subset T'[p]$,
showing that $T'[p]$ is a $Y$-modules.
Indeed, $T'[p]$ is irreducible and isomorphic to $M(1,p\gamma)$ for $p\in \Z_{>0}$ (cf.~\cite{DN}).
It also follows from the fusion rules \eqref{eq:4.24} that $\ker\,\pi$ is a $Y$-submodule.

Finally, we prove that $Y\oplus T'$ forms a subVOA of $U$ isomorphic to $V_{\Z\gamma}^+$.
Seen as a $Y\cong M(1)^+$-module, we have the following decomposition.
\[
  Y\oplus T'=M(1)^+\oplus \bigoplus_{p\in \Z_{>0}}M(1,p\gamma).
\]
Recall some fusion rules of $M(1)^+$-modules from \cite{A}.
Let $p$ and $q$ be positive integers such that $p\geq q$.
Then $M(1,p\gamma)\cong M(1,-p\gamma)$ and $M(1,q\gamma)\cong M(1,-q\gamma)$ as $M(1)^+$-modules
and we have the following fusion rules (cf.~Theorem 4.5 of \cite{A}\footnote{%
In the case $p=q$, $M(1,0)$ appears which we will understand as $M(1,0)=M(1)^+ + M(1)^-$.}).
\begin{equation}\label{eq:4.25}
  M(1,p\gamma)\times M(1,q\gamma)
  = M(1,(p+q)\gamma)+ M(1,(p-q)\gamma).
\end{equation}
Although $\tilde{T}_1=\ker\,\pi$ may not be a semisimple $Y$-module, we can apply
the fusion rules \eqref{eq:4.25} to deduce that $(T'\cd T')\cap \ker\,\pi=0$
by considering the filtration \eqref{eq:4.21} obtained from \eqref{eq:4.15}.
Each factor $\tilde{T}_i/\tilde{T}_{i+1}$ is a semisimple $Y$-module and there is no
simple subquotient of $\tilde{T}_1=\ker\,\pi$ isomorphic to neither $M(1)^\pm$ nor
$M(1,p\gamma)$ for $p\in \Z_{>0}$.
Therefore, the fusion rules \eqref{eq:4.25} of $Y$-modules guarantee
$(T'\cd T')\cap \ker\,\pi=0$ and hence $Y\oplus T'$ is a subVOA of $U$ isomorphic to
$V_{\Z\gamma}^+$.
\qed

\subsection{Simplicity of $V^{(n)}$}
We are now in a position to state the main result of this section.

\begin{thm}\label{thm:4.14}
  The vertex operator algebra $V^{(n)}$ is simple and isomorphic to $V_{\sqrt{2}A_{n-1}}^{+}$.
\end{thm}

\pf
We have shown in Lemma \ref{lem:4.13} that $V^{(n)}$ contains a conformal subalgebra isomorphic
to $V_{\sqrt{2}A_{n-2}}^+\tensor V_{\Z\gamma}^+$.
Recall the composition series \eqref{eq:4.15}.
Since the top factor $J_0/J_1$ is isomorphic to $V_{\sqrt{2}A_{n-1}}^+$ and $V^{(n)}$ is
of OZ-type, it follows that $V^{(n)}$ contains an irreducible
$V_{\sqrt{2}A_{n-2}}^+\tensor V_{\Z\gamma}^+$-submodule $V_{\sqrt{2}A_{n-2}}^-\tensor V_{\Z\gamma}^-$
with multiplicity one and thus $V^{(n)}$ contains the simple current extension
\[
  V_{\sqrt{2}A_{n-2}\oplus \Z\gamma}^+
  = V_{\sqrt{2}A_{n-2}}^+\tensor V_{\Z\gamma}^+\oplus V_{\sqrt{2}A_{n-2}}^-\tensor V_{\Z\gamma}^-
\]
of $V_{\sqrt{2}A_{n-2}}^+\tensor V_{\Z\gamma}^+$ as a conformal subalgebra.
Since $\alpha_{n-1}=2n\gamma^*-2\Lambda_{n-2}$ by \eqref{eq:4.1},
one has the following coset decomposition.
%\begin{equation}\label{eq:4.26}
\[
  \sqrt{2}A_{n-1}
  = \bigsqcup_{0\leq i<n-1} (\sqrt{2}A_{n-2}+2i\Lambda_{n-2}) \oplus (\Z\gamma-2ni\gamma^*).
\]
%\end{equation}
For brevity, set $\Gamma :=\sqrt{2}A_{n-2}\oplus \Z\gamma$ and
%\begin{equation}\label{eq:4.27}
\[
  \Gamma(i) :=(\sqrt{2}A_{n-2}+2i\Lambda_{n-2}) \oplus (\Z\gamma-2ni\gamma^*)
  ~~~\mbox{for}~~~ 0\leq i< n-1.
\]
%\end{equation}
Since $V_\Gamma^+$ is rational \cite{DJL}, there is a complement $V_\Gamma^+$-module $Z$
of $J_1$ in $V^{(n)}$ such that
%\begin{equation}\label{eq:4.28}
\[
  V^{(n)}=Z\oplus J_1
\]
%\end{equation}
and $Z\cong V_{\sqrt{2}A_{n-1}}^+$ as a $V_\Gamma^+$-module.
Let $\mathscr{A}$ be the set of inequivalent irreducible $V_\Gamma^+$-submodules of $Z$.
Then by \cite{AD}
%\begin{equation}\label{eq:4.29}
\[
  \mathscr{A}
  = \l\{ V_{\Gamma(0)}^+,~ V_{\Gamma(i)} ~\Big|~ 1\leq i <\dfr{n-1}{2} \r\}
\]
%\end{equation}
if $n$ is even, and
%\begin{equation}\label{eq:4.30}
\[
  \mathscr{A}
  = \l\{ V_{\Gamma(0)}^+,~ V_{\Gamma((n-1)/2)}^+,~ V_{\Gamma(i)} ~\Big|~ 1\leq i <\dfr{n-1}{2} \r\}
\]
%\end{equation}
if $n$ is odd, and the complement $Z$ has the following shape as a
$V_\Gamma^+$-module.
%\begin{equation}\label{eq:4.31}
\[
  Z
  = \bigoplus_{A\in \mathscr{A}} A.
\]
%\end{equation}

%Now let $\mathscr{F}(V_\Gamma^+)$ be the fusion ring, i.e., the Grothendieck ring of
%the tensor category of $V_\Gamma^+$-modules.
%Let $\mathscr{M}$ be the set of elements $[M]$ of $\mathscr{F}(V_\Gamma^+)$ where
%$M$ runs over $V_{\Gamma}^+$-submodules of $V^{(n)}$.
%We consider an action of $\mathscr{A}$ on $\mathscr{M}$ as follows.
%For $[A]\in \mathscr{A}$ and $[M]\in \mathscr{M}$, we pick the unique irreducible
%$V_{\Gamma}^+$-submodule $A$ of the complement $Z$ with given isomorphism class $[A]\in \mathscr{A}$,
%and then define $[A]\cd [M]:=[A\cd M]$.

Since $V^{(n)}$ is of OZ-type, there is no $V_{\Gamma}^+$-submodule isomorphic to
$V_{\Gamma(0)}^-$ in $V^{(n)}$. In the case that $n$ is odd, it follows from the fusion rules
\[
  V_{\Gamma((n-1)/2)}^+ \times V_{\Gamma((n-1)/2)}^- = V_{\Gamma(0)}^-
\]
of $V_{\Gamma}^+$-modules (cf.~\cite{ADL}) that $V^{(n)}$ does not contain an irreducible
$V_{\Gamma}^+$-submodule isomorphic to $V_{\Gamma((n-1)/2)}^-$, either.
Therefore, if there is an irreducible $V_\Gamma^+$-submodule $M$ of $J_1$ whose isomorphism class
belongs to $\mathscr{A}$, then every irreducible $V_\Gamma^+$-submodule of $Z\cd M$ is isomorphic
to a module in $\mathscr{A}$ by the fusion rules of $V_\Gamma^+$-modules (loc.~cit.).
Since the simple quotient $V^{(n)}/J_1\cong V_{\sqrt{2}A_{n-1}}^+$ is the only irreducible
$V^{(n)}$-module of $\sigma$-type which contains an irreducible $V_\Gamma^+$-submodule in $\mathscr{A}$,
we see that there is no irreducible $V_\Gamma^+$-submodule of $J_1$ whose
isomorphism class is in $\mathscr{A}$.
Therefore, we see that $Z\cd Z=Z$ again by the fusion rules of $V_\Gamma^+$-modules, and hence
$Z$ forms a subVOA of $V^{(n)}$.
Since the weight two subspace of $V^{(n)}$ coincides with that of $Z$ and $V^{(n)}$ is generated
by its Griess algebra, we conclude that $V^{(n)}=Z\cong V_{\sqrt{2}A_{n-1}^+}$.
\qed

\section{Simplicity of $V$ for the other cases} \label{sec:4}

%We consider the simplicity of a VOA $V$ satisfying Condition 1 whose simple quotient
%is not isomorphic to either $K(A_n,2)$ or $V_{\sqrt{2}A_n}^+$.
\begin{nota}
Let $V$ be a VOA satisfying Condition \ref{cond:1} and let $\mathcal{E}$ be a finite $\sigma$-closed generating set. Set $I_G=\{ \sigma_e \mid e\in \mathcal{E}\}$ and $G=\la I_G \ra$.
\end{nota}
\begin{comment}
Then $G$ is a finite indecomposable 3-transposition group and
the Griess algebra of $V$ is isomorphic to a homomorphic image of the Matsuo algebra
$B(G)$ associated with $G$.
Since $V$ is of OZ-type, it has the unique maximal ideal, and it follows from (4) of
Theorem \ref{thm:3.7} that the isomorphism class of the simple quotient of $V$ is uniquely
determined by $G$.
\end{comment} 
In this section, we will show that if $G$ is one of the groups listed in Theorem \ref{thm:main},  then the Griess algebra of $V$ is isomorphic to the non-degenerate quotient of $B(G)$, 
and $V$ is simple. Moreover, the structure of $V$ is uniquely determined by the group $G$.
%In this case the full group $G_V$ is also finite and uniquely determined.

By Lemma \ref{lem:6.1}, we will divide the proof into two cases: (A) $G$ is a group listed in Theorem \ref{thm:main} and $B(G)$ is non-degenerate; (B) $G$ is a group listed in Theorem \ref{thm:main}  but $B(G)$ is degenerate. 

%The main idea is to use induction on the size of the group $G$. 
The main strategy is as follows. 
We consider a  non-trivial proper indecomposable 3-transposition subgroup $H$ of $G$
where the set of 3-transpositions of $H$ is given by $I_H=H\cap I_G$.
Let $\mathcal{E}'=\{ e\in \mathcal{E} \mid \sigma_e \in H\}$ and consider the subVOA
$W=\la \mathcal{E}'\ra$ of $V$. By induction, we may assume the structure of $W$, i.e., $(H,W)$ is one of the pairs listed in Theorem \ref{thm:main}. We will argue
that $V$ is indeed simple by using our knowledge  about $W$ and its irreducible modules. The inequality in Lemma \ref{lem:3.14} will be used frequently.

Since the cases $G=\mathfrak{S}_n$ ($n\geq 3$) and $G=F_n{:}\mathfrak{S}_n$ ($n\geq 4$) have been proved by \cite[Theorem 4.1]{JLY} and Theorem \ref{thm:4.14}, we only consider the remaining  cases. 

\paragraph{Case A:} $G=\mathrm{O}_{6}^-(2)$, $\mathrm{Sp}_{6}(2)$, or $\mathrm{O}_{8}^+(2)$. 

\begin{thm}\label{thm:5.2}
	%Let $V$ be a VOA satisfying Condition \ref{cond:1} and take a finite $\sigma$-closed subset $\mathcal{E}$ of $E_V$ such that $V=\la \mathcal{E}\ra$ and $V_2=\C \mathcal{E}$. Set $G=\la \sigma_e\mid e\in \mathcal{E}\ra$.
Let $G = \mathrm{O}_6^-(2)$, $\mathrm{Sp}_6(2)$ or $\mathrm{O}_8^+(2)$.  Then  $V$ is simple.
\end{thm}

\pf  By Lemma \ref{lem:6.1}, the algebra $B(G)$ is non-degenerate for any of the groups above. 
Therefore, $B(G)$ is isomorphic to the Griess algebra of $V$ and that of its simple quotient. 
By Table 2, $B(G)$ is isomorphic to the Griess algebra of $K(E_6,2)$ (resp., $K(E_7,2)$  and  $K(E_8,2)$)  if $G=\mathrm{O}_{6}^-(2)$ (resp.,  $\mathrm{Sp}_{6}(2)$ and $\mathrm{O}_{8}^+(2)$) and we can identify their Ising vectors. 

Let  $J$ be the maximal ideal of $V$ and assume $V$ is not simple. Then the top weight $h$ of $J$ is greater than or equal to $3$ since the Griess algebras of $V$ and its simple quotient are isomorphic. 
 
\paragraph{Case 1: $G=\mathrm{O}_6^-(2)$.}
In this case,  we may assume  $\mathcal{E}=E_V=\{ w^-(\alpha) \mid \alpha\in \Delta_{E_6}^+\}$.
Let $\mathcal{E}'=\{ w^{-}(\alpha)\in \mathcal{E}\mid \alpha \in \Delta_{A_{5}}^+\}$. Then $H=\la \sigma_e \mid e\in \mathcal{E}'\ra \cong \mathfrak{S}_6$; thus,  $W=\la \mathcal{E}'\ra \cong K(A_5,2)$ by \cite[Theorem 4.1]{JLY}. 
By \cite{JL}, $K(A_5,2)$  has four
irreducible modules of $\sigma$-type with top weights $0$, $3/4$, $5/4$ and $3/2$.

Applying Lemma \ref{lem:3.14} with $I_G=\abs{36}$, $I_H=\abs{15}$, $k_G=20$, $k_H=8$ and
$\lambda=3/2$, we obtain $h\leq 72/35<3$, a contradiction.
Therefore $J=0$ and $V$ is simple and isomorphic to $K(E_6,2)$
by Theorem \ref{thm:3.7}. 

\paragraph{Case 2: $G=\mathrm{Sp}_6(2)$.}
In this case,  $\mathcal{E}=E_V=\{ w^-(\alpha) \mid \alpha\in \Delta_{E_7}^+\}$. 
Let $\mathcal{E}'=\{ w^{-}(\alpha)\in \mathcal{E}\mid \alpha \in \Delta_{A_{7}}^+\}$.
Then $H=\la \sigma_e \mid e\in \mathcal{E}'\ra \cong \mathfrak{S}_8$.   
By \cite[Theorem 4.1]{JLY}, $W=\la \mathcal{E}'\ra \cong K(A_7,2)$,  which has five irreducible modules
of $\sigma$-type with top weights $0$, $4/5$, $7/5$, $9/5$ and $2$ by \cite{JL}.

Applying Lemma \ref{lem:3.14} with $\abs{I_G}=63$, $\abs{I_H}=28$, $k_G=32$, $k_H=12$ and
$\lambda=2$, we obtain $h\leq 9/4<3$, a contradiction.
Therefore $J=0$ and $V$ is simple and isomorphic to $K(E_7,2)$
by Theorem  \ref{thm:3.7}. 

\paragraph{Case 3: $G=\mathrm{O}_8^+(2)$.}
In this case, we take $\mathcal{E}=\{ w^-(\alpha) \mid \alpha\in \Delta_{E_8}\}$.
Take a subsystem $\Delta_{E_7}$ of $\Delta_{E_8}$ and let
$\mathcal{E}'=\{ w^-(\alpha)\in \mathcal{E} \mid \alpha\in \Delta_{E_7}\}$,
$W=\la \mathcal{E}'\ra$ and
$H=\la \sigma_e\mid e\in \mathcal{E}'\ra\cong \mathrm{Sp}_6(2)$.
By Case 2, $W$ is simple and isomorphic to $K(E_7,2)$. 
As we have already seen in Case 2, the maximum top weight of irreducible $K(E_7,2)$-modules
of $\sigma$-type is less than or equal to $9/4$.

Applying Lemma \ref{lem:3.14} with $\abs{I_G}=120$, $\abs{I_H}=63$, $k_G=56$, $k_H=32$ and
$\lambda=9/4$, we obtain $h\leq 75/28<3$, a contradiction.
Therefore $J=0$ and $V$ is simple and isomorphic to $K(E_8,2)$
by Theorem \ref{thm:3.7}.
\medskip

\paragraph{Case B:} $G =F^2_n{:}\mathfrak{S}_n~(n\geq 4)$, $2^6{:}\mathrm{O}_6^-(2)$,  $2^6{:}\mathrm{Sp}_6(2)$, 
  $\mathrm{O}_8^-(2)$, $\mathrm{Sp}_8(2)$, $2^8{:}\mathrm{O}_8^+(2)$, or $\mathrm{O}_{10}^+(2)$.

\begin{thm}\label{thm:5.1}
%Let $V$ be a VOA satisfying Condition \ref{cond:1} and take a finite $\sigma$-closed subset $\mathcal{E}$ of $E_V$ such that $V=\la \mathcal{E}\ra$ and $V_2=\C \mathcal{E}$. Set $G=\la \sigma_e\mid e\in \mathcal{E}\ra$.
Let $G= F_n^2{:}\mathfrak{S}_n~(n\geq 4)$. Then the Griess algebra of $V$ is isomorphic
to the non-degenerate quotient of $B(G)$ and $V$ is simple.
\end{thm}

\pf We will prove the theorem by induction on $n$, $n\geq 4$.  
Note that 
the non-degenerate quotient of $B(G)$ is isomorphic to the Griess algebra of $V_{\sqrt{2}D_n}^+$ when $G= F_n^2{:}\mathfrak{S}_n~(n\geq 4)$.

First we consider the case $n=4$, i.e., $G\cong F^2_4{:}\mathfrak{S}_4$.
Suppose $V$ is not simple. Then the maximal ideal $J$ has the top weight greater than or equal to $2$. 

Take a 3-transposition subgroup $H\cong F_4{:}\mathfrak{S}_4$ of $G$ and consider the subVOA
$W=\la \mathcal{E}'\ra$ with $\mathcal{E}'=\{ e\in \mathcal{E} \mid \sigma_e\in H\}$.
Then $W$ is simple and isomorphic to $V_{\sqrt{2}A_3}^+$ by Lemma \ref{lem:4.7}.
It follows from Remark \ref{rem:4.2} that the top weights of irreducible $W$-modules
of $\sigma$-type are  less than and equal to $1$.
Applying Lemma \ref{lem:3.14} with $\abs{I_G}=24$, $\abs{I_H}=12$, $k_G=16$, $k_H=8$ and
$\lambda=1$, we obtain $h\leq 4/3<2$, a contradiction.
Therefore $J=0$ and $V$ is simple.

Now assume that $V$ is simple if $G=F^2_n{:}\mathfrak{S}_n$ for some $n\geq 4$ and consider the case $G=F^2_{n+1}{:}\mathfrak{S}_{n+1}$.

We will first show that the Griess algebra of $V$ is isomorphic to that of $V_{\sqrt{2}D_n}^+$.  
Suppose false. The bilinear form on the Griess algebra of $V$ is degenerate. Let  
$J_2$ be the radical of the Griess algebra of $V$. 
Note that the algebra $B(F_n^2{:}\mathfrak{S}_n)$ has the dimension $\abs{I_G}=2n(n-1)$ and 
the non-degenerate quotient of $B(F_n^2{:}\mathfrak{S}_n)$
is realized by the Griess algebra of $V_{\sqrt{2}D_n}^+$ which is of dimension $n(3n-1)/2$ (cf.~Table 1).
Therefore, the radical of $B(F_n^2{:}\mathfrak{S}_n)$ has the dimension
$2n(n-1)-n(3n-1)/2=n(n-3)/2$.

Take a 3-transposition subgroup $H\cong F_n^2{:}\mathfrak{S}_n$ of $G$ and consider the subVOA
$W=\la \mathcal{E}'\ra$ with $\mathcal{E}'=\{ e\in \mathcal{E} \mid \sigma_e\in H\}$.
Then $W$ is simple by inductive assumption and the bilinear form of $V$ restricted to
the Griess algebra of $W$ is non-degenerate.
Therefore, we have the following bound on the dimension of the radical.
\[
  \dim J_2\leq \dfr{(n+1)(n+1-3)}{2}-\dfr{n(n-3)}{2}=n-1.
\]
Consider another 3-transposition subgroup $K=F_{n+1}{:}\mathfrak{S}_{n+1}$ of $G$.
Then the subVOA $U=\la \mathcal{E}''\ra$ with
$\mathcal{E}''=\{ e\in \mathcal{E} \mid \sigma_e\in K\}$ is isomorphic to $V_{\sqrt{2}A_n}^+$
by Theorem \ref{thm:4.14}.
The irreducible $V_{\sqrt{2}A_n}^+$-modules of $\sigma$-type are as described in the proof
of Lemma \ref{lem:4.1} and it follows from Remark \ref{rem:4.2} that the top level of
an irreducible $V_{\sqrt{2}A_n}^+$-module of $\sigma$-type has dimension at least $n$
if it is not isomorphic to the adjoint module $V_{\sqrt{2}A_n}^+$.
Therefore, the zero-mode of the conformal vector of $U$ acts on the radical $J_2$ as 0.
Then since $G$ is indecomposable, the conformal vector of $V$ also acts on $J_2$ as 0
by \eqref{eq:2.3} and Lemma \ref{lem:3.13} which is absurd.
Therefore $J_2=0$ and the Griess algebra of $V$ is non-degenerate; thus, the Griess algebra of $V$ is isomorphic to that of $V_{\sqrt{2}D_n}^+$. 

In this case, we may assume  $\mathcal{E}=E_V=\{w^\pm(\alpha)\mid \alpha\in \Delta_{D_n}^+\}$
\cite{LSY}.
%and $G=G_V=F^2_n{:}\mathfrak{S}_n$.
Let $\epsilon_1,\dots,\epsilon_n$ be an orthonormal basis of $\R^n$, i.e.,
$(\epsilon_i\mymid \epsilon_j)=\delta_{ij}$. A root system of type $D_n$ is given by 
$
  \Delta_{D_n}=\{ \pm \epsilon_i\pm \epsilon_j \mid 1\leq i<j\leq n\}.
$ 
We fix a set of simple roots $\{ \alpha_1,\dots,\alpha_n\}$ of $\Delta_{D_n}$
so that $\alpha_i=\epsilon_i-\epsilon_{i+1}$ for $1\leq i<n$ and
$\alpha_n=\epsilon_{n-1}+\epsilon_n$.

Let $\alpha_0=\epsilon_1+\epsilon_2$ be the highest root of $\Delta_{D_n}$.
Set $ \Delta_0=\{  \pm \alpha_0 \}$, $\Delta_1=\{ \pm \alpha_1\}$ and
$\Delta_2=\{ \pm \epsilon_i\pm \epsilon_j \mid 3\leq i<j\leq n\}$.
Then $\Delta_0\cong \Delta_1\cong \Delta_{A_1}$ and $\Delta_2\cong \Delta_{D_{n-2}}$
where we identify $D_2=A_1\oplus A_1$ and $D_3=A_3$.
We also set $\Delta'=\Delta_0 \sqcup \Delta_1\sqcup \Delta_2$.
Let $U^i=\la w^\pm(\alpha) \mid \alpha \in \Delta_i\ra$ for $i=0,1,2$.
Then $U^0\cong U^1\cong L(\shf,0)\tensor L(\shf,0)$.
By inductive assumption,  $U^2\cong V_{\sqrt{2}D_{n-2}}^+$.
Note that
$(V_{\sqrt{2}A_1}^+)^{\tensor 2} \cong L(\shf,0)^{\tensor 4}$ and
$V_{\sqrt{2}D_3}^+\cong V_{\sqrt{2}A_3}^+$.
Clearly, $\la w^\pm (\alpha) \mid \alpha \in \Delta'\ra=U^0\tensor U^1\tensor U^2$.
Let
\[
  x=\sum_{1\leq i\leq 2\atop 3\leq j\leq n}
  (-1)^i \l( w^-(\epsilon_i-\epsilon_j)+w^-(\epsilon_i+\epsilon_j)
  +w^+(\epsilon_i-\epsilon_j)+w^+(\epsilon_i+\epsilon_j)\r) .
\]
It follows from \eqref{eq:inner} and \eqref{eq:prod} that $2w^\pm(\alpha)_{(1)}x=x$ for
$\alpha\in \Delta_0\sqcup \Delta_1$ and $w^\pm(\beta)_{(1)}x=0$ for $\beta\in \Delta_2$.
Thus the $U^0\otimes U^1\cong L(\shf,0)^{\tensor 4}$-submodule $X$ of $V$ generated by $x$ is
isomorphic to $L(\shf,\shf)^{\otimes 4}$ and the simple current extension
$U^0\tensor U^1\oplus X$ is isomorphic to $V_{\sqrt{2}(A_1\oplus A_1)}^+$
by Proposition 4.1 of \cite{LSY} which is in the commutant of $U^2$ in $V$.
Set
\[
   W = (U^0\otimes U^1\oplus X)\otimes U^{2}\cong V_{\sqrt{2}(A_1\oplus A_1)}^+\tensor V_{\sqrt{2}D_{n-2}}^+.
 \]
Then $W$ is a rational conformal subVOA of $V$ by \cite{DJL}.
Let $\gamma=(\alpha_0+\alpha_1)/2$ and $\delta=(\alpha_{n-1}+\alpha_n)/2$.
One has the following decomposition of $V_{\sqrt{2}D_n}^+$ as a
$V_{\sqrt{2}(A_1\oplus A_1)}^+\tensor V_{\sqrt{2}D_{n-2}}^+$-module.
\begin{equation}\label{eq:5.3}
\begin{split}
  V_{\sqrt{2}D_{n}}^+
  \cong~
  & \l( V_{\sqrt{2}(A_1\oplus A_1)}^+\tensor V_{\sqrt{2}D_{n-2}}^+\r)
  \oplus \l( V_{\sqrt{2}(A_1\oplus A_1)}^-\tensor V_{\sqrt{2}D_{n-2}}^-\r)
  \medskip\\
  & \oplus \l( V_{\sqrt{2}(A_1\oplus A_1)+\gamma}^+\tensor V_{\sqrt{2}D_{n-2}+\delta}^+\r)
  \oplus \l( V_{\sqrt{2}(A_1\oplus A_1)+\gamma}^-\tensor V_{\sqrt{2}D_{n-2}+\delta}^-\r) .
\end{split}
\end{equation}
Let $J$ be the maximal ideal of $V$.
Then the simple quotient $V/J$ is isomorphic to $V_{\sqrt{2}D_n}^+$.
Since $W$ is rational, there is a $W$-submodule $M$ of $V$ such that $V=M\oplus J$
and $M\cong V_{\sqrt{2}D_n}^+$.
By \eqref{eq:5.3}, there exist irreducible $W$-submodules $W^0$, $W^1$, $W^2$ and $W^3$
of $V$ such that $M = W^0 \oplus W^1 \oplus W^2 \oplus W^3$ as a $W$-module, where 
\[
\begin{array}{l}
  W^0 = W\cong V_{\sqrt{2}(A_1\oplus A_1)}^+\tensor V_{\sqrt{2}D_{n-2}}^+,~~~
  W^1\cong V_{\sqrt{2}(A_1\oplus A_1)}^-\tensor V_{\sqrt{2}D_{n-2}}^-,
  \medskip\\
  W^2\cong V_{\sqrt{2}(A_1\oplus A_1)+\gamma}^+\tensor V_{\sqrt{2}D_{n-2}+\delta}^+,~~~
  W^3\cong V_{\sqrt{2}(A_1\oplus A_1)+\gamma}^-\tensor V_{\sqrt{2}D_{n-2}+\delta}^-.
\end{array}
\]
Since the Griess algebra of $V$ and that of the simple quotient $V/J\cong V_{\sqrt{2}D_n}^+$
are isomorphic, the top weight of $J$ is at least 3 and hence there is no irreducible
$W$-submodule of $J$ isomorphic to $W^i$ for $0\leq i\leq 3$.
Then it follows from the fusion rules of $W$-modules (cf.~\cite{ADL}) that $M$ forms a
simple subVOA of $V$ which is a simple current extension of $W$.
Since $V_2=M_2$, all the Ising vectors $w^\pm(\alpha)\in E_V$ are contained in the subVOA $M$.
Therefore, we have $V=M$ and $J=0$ as $V$ is generated by $E_V$.
\qed

\begin{thm}\label{lem:6.3}
  Suppose $G$ is one of the following.
  \[
    2^{6}{:}\mathrm{O}_{6}^-(2),~~~
    2^{6}{:}\mathrm{Sp}_{6}(2),~~~
    \mathrm{O}_8^-(2),~~~
    \mathrm{Sp}_8(2),~~~
    2^8{:}\mathrm{O}_8^+(2),~~~
    \mathrm{O}_{10}^+(2).
  \]
  Then $V$ is simple and the Griess algebra of $V$ is isomorphic to the non-degenerate quotient of $B(G)$.
\end{thm}

\pf
If $V$ is not simple, then the maximal ideal $J$ of $V$ has the top weight greater than or equal to $2$ since $V$ is of OZ-type. In this case, the weight two subspace $J_2$ of $J$ coincides with the radical of the bilinear form on the Griess algebra of $V$.

\paragraph{Case 1:~$G=2^{6}{:}\mathrm{O}_6^-(2)$.}
Take a 3-transposition subgroup $H\cong F_6{:}\mathfrak{S}_6$ of $G$ and consider
the subVOA $W=\la \mathcal{E}'\ra$ with $\mathcal{E}'=\{ e\in \mathcal{E} \mid \sigma_e\in H\}$.
Then $W\cong V_{\sqrt{2}A_5}^+$ by Theorem \ref{thm:4.14}.
By Remark \ref{rem:4.2}, the top weights of irreducible $W$-modules of $\sigma$-type are less than or equal to $3/2$.
Applying Lemma \ref{lem:3.14} with $\abs{I_G}=72$, $\abs{I_H}=30$, $k_G=40$, $k_H=16$
and $\lambda=3/2$, we obtain $h\leq 9/5\lneq 2$, a contradiction.
Therefore $J=0$ and $J_2=0$; thus $V$ is simple and isomorphic to $V_{\sqrt{2}E_6}^+$ by Theorem \ref{thm:3.7}.

\paragraph{Case 2:~$G=2^{6}{:}\mathrm{Sp}_6(2)$.}
Take a 3-transposition subgroup $H\cong 2^6{:}\mathrm{O}_6^-(2)$ of $G$ and consider
the subVOA $W=\la \mathcal{E}'\ra$ with $\mathcal{E}'=\{ e\in \mathcal{E} \mid \sigma_e\in H\}$.
Then $W\cong V_{\sqrt{2}E_6}^+$ by Case 1.
Since the irreducible modules of $V_{\sqrt{2}E_6}$ are classified \cite{AD}, it can be verified directly that the top weights of irreducible $W$-modules are less than or equal to $4/3$.

Applying Lemma \ref{lem:3.14} with $\abs{I_G}=126$, $\abs{I_H}=72$, $k_G=64$, $k_H=40$
and $\lambda=4/3$, we obtain $h\leq 14/9 < 2$, a contradiction.
Therefore $J=0$ and $V$ is simple and isomorphic to $V_{\sqrt{2}E_7}^+$ by Theorem \ref{thm:3.7}.

\paragraph{Case 3:~$G=\mathrm{O}_8^-(2)$.}
Take a 3-transposition subgroup $H\cong 2^6{:}\mathrm{O}_6^-(2)$ of $G$ and consider
the subVOA $W=\la \mathcal{E}'\ra$ with $\mathcal{E}'=\{ e\in \mathcal{E} \mid \sigma_e\in H\}$.
Then $W\cong V_{\sqrt{2}E_6}^+$ by Case 1.
As we have seen in the proof of Case 2, the top weights of
irreducible $W$-modules are less than or equal to $4/3$.

Applying Lemma \ref{lem:3.14} with $\abs{I_G}=136$, $\abs{I_H}=72$, $k_G=72$, $k_H=40$
and $\lambda=4/3$, we obtain $h\leq 68/45<2$, a contradiction.
Therefore $J=0$ and $V$ is simple and isomorphic to
$\com(K(A_2,2),V_{\sqrt{2}E_8}^+)$ by Theorem \ref{thm:3.7}.

\paragraph{Case 4:~$G=\mathrm{Sp}_8(2)$.}
Take a 3-transposition subgroup $H\cong \mathrm{O}_8^+(2)$ of $G$ and consider
the subVOA $W=\la \mathcal{E}'\ra$ with $\mathcal{E}'=\{ e\in \mathcal{E} \mid \sigma_e\in H\}$.
Then $W\cong K(E_8,2)$ by Theorem \ref{thm:5.2}.
Since $K(E_8,2)$ is a code VOA (cf.~\cite{LSY}), one can directly verify that there exist
exactly two inequivalent irreducible $K(E_8,2)$-modules of $\sigma$-type,
$M^0\cong K(E_8,2)$ and $M^1$ with the top weight $3/2$.
Indeed, we have a decomposition $V_{\sqrt{2}E_8}^+=L(\shf,0)\tensor M^0\oplus L(\shf,\shf)\tensor M^1$.
Therefore, the maximum top weight of irreducible $W$-modules of $\sigma$-type is $3/2$.

Applying Lemma \ref{lem:3.14} with $\abs{I_G}=255$, $\abs{I_H}=120$, $k_G=128$, $k_H=56$
and $\lambda=3/2$, we obtain $h\leq 3/2<2$, a contradiction.
Therefore $J_2=0$ and $V$ is isomorphic to $K(E_8,2)$ by Theorem \ref{thm:3.7}. 

\paragraph{Case 5:~$G=2^8{:}\mathrm{O}_8^+(2)$.}
Take a 3-transposition subgroup $H\cong 2^6{:}\mathrm{Sp}_6(2)$ of $G$ and consider
the subVOA $W=\la \mathcal{E}'\ra$ with $\mathcal{E}'=\{ e\in \mathcal{E} \mid \sigma_e\in H\}$.
Then $W\cong V_{\sqrt{2}E_7}^+$ by Case 2.
It is easy to verify that the top weights of irreducible $W$-modules are $\leq 3/2$.

Applying Lemma \ref{lem:3.14} with $\abs{I_G}=240$, $\abs{I_H}=126$, $k_G=112$, $k_H=64$
and $\lambda=3/2$, we obtain $h\leq 12/7 <2$, a contradiction.
Therefore $J=0$ and $V$ is isomorphic to $V_{\sqrt{2}E_8}^+$ by Theorem \ref{thm:3.7}. 

\paragraph{Case 6:~$G=\mathrm{O}_{10}^+(2)$.}
%%%%%%%%%%%%%%%%%%%
Take a 3-transposition subgroup $H\cong 2^6{:}\mathrm{Sp}_6(2)$ of $G$ and consider
the subVOA $W=\la \mathcal{E}'\ra$ with $\mathcal{E}'=\{ e\in \mathcal{E} \mid \sigma_e\in H\}$.
Then $W\cong V_{\sqrt{2}E_7}^+$ by Case 2.
As we have seen,the top weights of irreducible $W$-modules are $\leq 3/2$.

Applying Lemma \ref{lem:3.14} with $\abs{I_G}=496$, $\abs{I_H}=126$, $k_G=240$, $k_H=64$
and $\lambda=3/2$, we obtain $h\leq 12/7$, a contradiction.
Therefore $J=0$ and $V$ is  simple.  
Since the non-degenerate quotient of $B(\mathrm{O}_{10}^+(2))$ is isomorphic to the non-degenerate quotient
of $B(2^8{:}\mathrm{O}_8^+(2))$ and both are realized as the Griess algebra of $V_{\sqrt{2}E_8}^+$, $V$ is also isomorphic to $V_{\sqrt{2}E_8}^+$ by Case 5 since it is generated by $V_2$. 
\qed
\medskip

\section{Classification}

So far, we have shown that if $G=\la \mathcal{E}\ra$ is one of the groups in Theorem \ref{thm:main}
then $V$ is simple and unique.
By Theorems \ref{thm:hall} and \ref{thm:3.7}, the following cases are left.
\begin{equation}\label{eq:6.3}
\hspace{-3pt}
\begin{array}{l}
  F_n^m{:}\mathfrak{S}_n~(n\geq 4,~m\geq 3),~~~
  (\mathbb{F}_2^6)^m{:}\mathrm{O}_{6}^-(2)~(m\geq 2),~~~
  (\mathbb{F}_2^6)^m{:}\mathrm{Sp}_{6}(2)~(m\geq 2),~~
  \medskip\\
  (\mathbb{F}_2^8)^m{:}\mathrm{O}_{8}^+(2)~(m\geq 2),~~~
  (\mathbb{F}_2^8)^m{:}\mathrm{O}_{8}^-(2)~(m\geq 1),~~~
  (\mathbb{F}_2^8)^m{:}\mathrm{Sp}_{8}(2)~(m\geq 1),~~
  \medskip\\
  (\mathbb{F}_2^{10})^m{:}\mathrm{O}_{10}^+(2)~(m\geq 1),~~
  (\mathbb{F}_2^{2n})^m{:}\mathrm{O}_{2n}^-(2)~(n\geq 5,~ m\geq 0),
  \medskip\\
  (\mathbb{F}_2^{2n})^m{:}\mathrm{Sp}_{2n}(2)~(n\geq 5,~ m\geq 0),~~~
  (\mathbb{F}_2^{2n})^m{:}\mathrm{O}_{2n}^+(2)~(n\geq 6,~ m\geq 0).
\end{array}
\end{equation}
In order to eliminate these groups, we will consider the Gram matrices of the bilinear forms on  $B(G)$.  
Let $(G,I)$ be a 3-transposition group. Then the Gram matrix of the algebra $B(G)=\oplus_{i\in I}\C x^i$ 
is defined to be the matrix  $( (x^i|x^j))_{i,j\in I}$. By \eqref{eq:2.1}, it is equal to \begin{equation}\label{eq:6.4}
  \dfr{1}{32}(8\mathbb{I}+A)
\end{equation}
where $\mathbb{I}$ is the identity matrix of size $\abs{I}$ and $A$ is the adjacency matrix
of the graph structure on $I$. If $G$ is $\mathfrak{S}_n$, $\mathrm{Sp}_{2n}(2)$ or $\mathrm{O}_{2n}^\pm(2)$, the graph structure on $I$ is a strongly regular graph $\mathrm{srg}(v,k,\lambda,\mu)$
with the following parameters (cf.~\cite{Ma}).

\[
\renewcommand{\arraystretch}{1.5}
\begin{array}{c}
\begin{array}{|c||c|c|c|c|c|c|}
  \hline
  G & v & k & \lambda & \mu & r & s
  \\ \hline
  \mathfrak{S}_n~(n\geq 4) & n(n-1)/2 & 2(n-2) & n-2 & 4 & n-4 & -2
  \\ \hline
  \mathrm{O}_{2n}^+(2) & 2^{2n-1}-2^{n-1} & 2^{2n-2}-2^{n-1} & 2^{2n-3}-2^{n-2} & 2^{2n-3}-2^{n-1}
  & 2^{n-1} & -2^{n-2}
  \\ \hline
  \mathrm{O}_{2n}^-(2) & 2^{2n-1}+2^{n-1} & 2^{2n-2}+2^{n-1} & 2^{2n-3}+2^{n-2} & 2^{2n-3}+2^{n-1}
  & 2^{n-2} & -2^{n-1}
  \\ \hline
  \mathrm{Sp}_{2n}(2) & 2^{2n}-1 & 2^{2n-1} & 2^{2n-2} & 2^{2n-2} & 2^{n-1} & -2^{n-1}
  \\\hline
\end{array}
\\
\mbox{Table 5: Parameters of strongly regular graphs}
\end{array}
\renewcommand{\arraystretch}{1}
\]
In the table above, $r$ and $s$ denote the larger and smaller eigenvalues of $A$ other than
the eigenvalue $k$, the valency of $I$.
Note that the adjacency matrix $A$ has exactly three eigenvalues, $r$, $s$ and $k$.
The adjacency matrices of the extensions $F_{n}^m{:}\mathfrak{S}_n$ and
$\mathbb{F}_{2n}^m{:}G$ with $G=\mathrm{Sp}_{2n}(2)$ or $\mathrm{O}_{2n}^\pm(2)$ are given by
$\mathbb{J}_{2^m}\tensor A$ where $A$ is the adjacency matrix of the group without extension and
$\mathbb{J}_{2^m}$ is the all-ones matrix of size $2^m$ (cf.~\cite{Ma}).

\begin{lem}\label{lem:6.4}
  For $m\geq 3$ and $n\geq 4$, $G=F_n^m{:}\mathfrak{S}_n$ is impossible.
\end{lem}

\pf
Let $m\geq 3$ and $n\geq 4$.
Since $F_4^3{:}\mathfrak{S}_4$ is a 3-transposition subgroup of $F_n^m{:}\mathfrak{S}_n$,
it suffices to show that $G=F_4^3{:}\mathfrak{S}_4$ is impossible.
Suppose there exists a VOA $V=\la \mathcal{E}\ra$ generated by a finite $\sigma$-closed subset
$\mathcal{E}$ of $E_V$ such that $V_2=\C \mathcal{E}$ and
$G=\la \sigma_e\mid e\in \mathcal{E}\ra\cong F_4^3{:}\mathfrak{S}_4$.
In this case the Gram matrix of the Matsuo algebra $B(G)$ is given by
$2^{-5}(8\mathbb{I}+\mathbb{J}_{8}\tensor A)$ where $A$ is the adjacency matrix of
the strongly regular graph $\mathrm{srg}(6,4,2,4)$.
The eigenvalues of $\mathbb{J}_{8}\tensor A$ are 32, 0 and $-16$ so that the Gram matrix
is non-singular.
Therefore, the Ising vectors in $\mathcal{E}$ are linearly independent.
Take a 3-transposition subgroup $H=F_4^2{:}\mathfrak{S}_4$ of $G$ and consider
the subVOA $W=\la \mathcal{E}'\ra$ with $\mathcal{E}'=\{ e\in \mathcal{E} \mid \sigma_e\in H\}$.
Then $W\cong V_{\sqrt{2}D_4}^+$ by Theorem \ref{thm:5.1}.
The dimension of the Griess algebra of $W$ is 22 but the size of $\mathcal{E}'$ is 24 so that
Ising vectors in $\mathcal{E}'$ are not linearly independent.
This contradicts that the Ising vectors in $\mathcal{E}$ are linearly independent.
Thus $G=F_4^3{:}\mathfrak{S}_4$ is impossible.
\qed

\begin{lem}\label{lem:6.5}
  The group $G$ cannot be one of the following:
 \begin{equation}\label{eq1}
 (2^6)^2{:}\mathrm{O}_{6}^-(2),~
 (2^6)^2{:}\mathrm{Sp}_6(2),~
 (2^8)^2 {:} \mathrm{O}_8^+(2),~
 2^8{:}\mathrm{O}_{8}^-(2),~
 2^8{:}\mathrm{Sp}_8(2),~
 2^{10}{:}\mathrm{O}_{10}^+(2),~
 \mathrm{O}_{10}^-(2).
\end{equation} 
\end{lem}

\pf
Let $G$ be one of the groups in \eqref{eq1}.
By considering the eigenvalues of the adjacency matrices, we note that
the Gram matrix of $B(G)$ is non-singular.
Suppose there exists a VOA $V=\la \mathcal{E}\ra$ such that $G=\la \mathcal{E}\ra$ and
$V_2=\C \mathcal{E}$.
Then the Griess algebra of $V$ is isomorphic to the algebra $B(G)$
and Ising vectors in $\mathcal{E}$ are linearly independent.
On the other hand, take a 3-transposition subgroup 
$H\cong 2^6{:}\mathrm{O}_6^-(2)$ if $G=(2^6)^2{:}\mathrm{O}_6^-(2)$, 
$H\cong 2^6{:}\mathrm{Sp}_6(2)$ if $G=(2^6)^2{:}\mathrm{Sp}_6(2)$,
$H\cong 2^8{:} \mathrm{O}_8^+(2)$ if $G=(2^8)^2 {:} \mathrm{O}_8^+(2)$, 
$H\cong \mathrm{O}_8^-(2)$ if $G=2^8 {:} \mathrm{O}_8^-(2)$ and 
$H\cong \mathrm{Sp}_8(2)$ if $G$ is either $2^8{:}\mathrm{Sp}_8(2)$, $2^{10}{:}\mathrm{O}_{10}^+(2)$ 
or $\mathrm{O}_{10}^-(2)$.
We consider the subVOA $W=\la \mathcal{E}'\ra$ with
$\mathcal{E}'=\{ e\in \mathcal{E} \mid \sigma_e\in H\}$.
Then $W$ is simple and uniquely determined by Lemma \ref{lem:6.3}.
Since $W_2=\C \mathcal{E}'$ is isomorphic to the non-degenerate quotient of the  algebra 
$B(H)$ which has a smaller dimension than the number of Ising vectors 
$\abs{\mathcal{E}'}=\abs{I_H}$ in any cases.
This contradicts that the Ising vectors in $\mathcal{E}$ are linearly independent since 
$\mathcal{E}'\subset \mathcal{E}$.
Thus $G$ cannot be one of the groups in the assertion.
\qed

\begin{cor}\label{cor:6.6}
  The group $G$ cannot be any of the groups listed in \eqref{eq:6.3}.
\end{cor}

\pf
By the inclusions
\[
  \mathrm{O}_{2n}^\pm (2)<\mathrm{Sp}_{2n}(2)<\mathrm{O}_{2n+2}^\mp(2)<\mathrm{Sp}_{2n+2}(2),~~
  \mathrm{Sp}_{2n}(2)<2^{2n}{:}\mathrm{Sp}_{2n}(2)<\mathrm{Sp}_{2n+2}(2),
\]
and $(\mathbb{F}_2^{2n})^m{:}H<(\mathbb{F}_2^{2n})^{m+1}{:}H$ for $H=\mathrm{Sp}_{2n}(2)$
or $H=\mathrm{O}_{2n}^\pm(2)$, those groups in the statement are forbidden by Lemmas
\ref{lem:6.4}, and \ref{lem:6.5}.
\qed
\medskip

Summarizing everything, we obtain the following complete classification of VOAs satisfying
Condition \ref{cond:1}.

\begin{thm}\label{thm:6.7}
  Let $V$ be a VOA satisfying Condition \ref{cond:1} and $E_V$ the set of Ising vectors of $V$.
  Suppose further that $E_V$ is indecomposable and $\abs{E_V}>1$.
  Take a finite $\sigma$-closed generating subset $\mathcal{E}$ of $E_V$.
  Then $G=\la \sigma_e \mid e\in \mathcal{E}\ra$ is isomorphic to one of the following:
  \[
  \begin{array}{l}
    \mathfrak{S}_n~ (n\geq 3),~~
    F_n{:}\mathfrak{S}_n~ (n\geq 4),~~
    F_n^2{:}\mathfrak{S}_n~ (n\geq 4),~~
    \mathrm{O}_6^-(2),~~
    2^6{:}\mathrm{O}_6^-(2),
    \medskip\\
    \mathrm{Sp}_6(2),~~
    2^6{:}\mathrm{Sp}_6(2),~~
    \mathrm{O}_8^+(2),~~
    \mathrm{O}_8^-(2),~~
    \mathrm{Sp}_8(2),~~
    2^8{:}\mathrm{O}_8^+(2),~~
    \mathrm{O}_{10}^+(2).
  \end{array}
  \]
  Moreover, $V$ is simple and the Griess algebra of $V$ is isomorphic to the non-degenerate
  quotient of the algebra $B(G)$.
  It turns out that $E_V$ is finite and $G_V=\la \sigma_e \mid e\in E_V\ra$ is a finite 3-transposition
  group, and $G=G_V$ except for $G=\mathrm{O}_8^+(2)$ or $G=2^8{:}\mathrm{O}_8^+(2)$, where
  $G_V=\mathrm{Sp}_8(2)$ if $G=\mathrm{O}_8^+(2)$ and $G_V=\mathrm{O}_{10}^+(2)$ if
  $G=2^8{:}\mathrm{O}_8^+(2)$.
  The real subalgebra $V_\R=\la E_V\ra_\R$ of $V$ is compact, i.e., the bilinear form on $V_\R$ with $(\vac \mymid \vac)=1$ is positive definite.
\end{thm}

Since a simple VOA satisfying Condition \ref{cond:1} is uniquely determined by its Griess algebra,
we have established Theorem \ref{thm:main}.

%%%%%%%%%%%%%%%%%%%%%%%%%%%%%%%%%%%%%%%%%%%%%%%%%%%%%%%%%%%%%%%%%%%%%%%

\small

\end{document}